\documentclass[10.5pt,reqno,a4paper]{amsart}

\usepackage{floatrow}
\usepackage{caption}
 \usepackage[foot]{amsaddr}
\let\oldtocsection=\tocsection
 
\let\oldtocsubsection=\tocsubsection 
 
\let\oldtocsubsubsection=\tocsubsubsection

\renewcommand{\tocsection}[2]{\vspace{0.5em}\hspace{0em}\oldtocsection{#1}{#2}}
\renewcommand{\tocsubsection}[2]{\vspace{0.5em}\hspace{1em}\oldtocsubsection{#1}{#2}}
\renewcommand{\tocsubsubsection}[2]{\vspace{0.5em}\hspace{2em}\oldtocsubsubsection{#1}{#2}}
\usepackage{graphicx,cancel,xcolor,hyperref,comment,graphicx,geometry}
 
\usepackage{amsopn}

\usepackage{tikz}
\usepackage{graphicx,url,etoolbox}
\usepackage{lipsum}
 \usepackage{amssymb}
 \newcommand{\isEquivTo}[1]{\underset{#1}{\sim}}
\usepackage{graphicx,cancel,xcolor,hyperref,comment,graphicx,geometry}

\usepackage{tikz}
\usetikzlibrary{decorations.pathreplacing}
\usepackage{graphicx,url,etoolbox}
\usepackage{lipsum}

\usepackage{fancyhdr}
\usepackage{lastpage}

\newtheorem{theoreme}{Theorem}

\newtheorem{rem}[theoreme]{Remark}
\theoremstyle{definition}

\usepackage{graphicx}
\usepackage{subcaption}

\numberwithin{equation}{section}

 \renewenvironment{proof}{{\bfseries \noindent Proof.}}{\demo}
\newcommand\xqed[1]{%
  \leavevmode\unskip\penalty9999 \hbox{}\nobreak\hfill
  \quad\hbox{#1}}
\newcommand\demo{\xqed{$\square$}}

\hypersetup{bookmarks, colorlinks, urlcolor=blue, citecolor=blue, linkcolor=blue, hyperfigures, pagebackref,
    pdfcreator=LaTeX, breaklinks=true, pdfpagelayout=SinglePage, bookmarksopen=true,bookmarksopenlevel=2}

\newtheorem*{nonamethm}{\nonamethmname}
\newcommand{\nonamethmname}{}

\newenvironment{genthm*}[1]
 {\renewcommand{\nonamethmname}{#1}\nonamethmcheck}
 {\endnonamethm}

\newcommand\nonamethmcheck[1][]{%
  \if\relax\detokenize{#1}\relax
    \nonamethm\relax
  \else
    \nonamethm[#1]%
  \fi
  \mbox{}%
}

\usepackage{comment}
\def\R{\mathbb R}

\def\la {{\lambda}}

\newcommand {\nc}   {\newcommand}
\nc {\be}   {\begin{equation}} \nc {\ee}   {\end{equation}} \nc
{\beq}  {\begin{eqnarray}} \nc {\eeq}  {\end{eqnarray}} \nc {\beqs}
{\begin{eqnarray*}} \nc {\eeqs} {\end{eqnarray*}}
\def\edc{\end{document}}

\providecommand{\abs}[1]{\lvert#1\rvert}
   
\usepackage{tikz}
\usetikzlibrary{decorations.pathmorphing,patterns,scopes,intersections,calc}
\usepackage{caption}
\usepackage{float}
\usepackage{soul}

\newcounter{dummy} 
\numberwithin{dummy}{section}
\newtheorem{Theorem}[dummy]{Theorem}

\newtheorem{Lemma}[dummy]{Lemma}

\newtheorem{Proposition}[dummy]{Proposition}
\newtheorem{Remark}[dummy]{Remark}

\numberwithin{equation}{section}

\newtheoremstyle{break}
  {\topsep}{\topsep}%
   {\itshape}{}%
  {\bfseries}{}%
  {\newline}{}%
\theoremstyle{break}

\begin{document}
\title[\fontsize{7}{9}\selectfont  ]{Asymptotic Behavior of a transmission  Heat/Piezoelectric smart material with internal fractional dissipation law}
 \author{Ibtissam Issa$^1$ , Abdelaziz Soufyane$^2$ , Octavio Vera Villagran$^3$ }
 
 \address{$^1$ Universit\'a  degli studi di Bari Aldo Moro-Italy, Dipartimento di Matematica, Via E. Orabona 4, 70125 Bari - Italy.}

\address{$^2 $ 	Department of Mathematics, College of Science, University of Sharjah, P.O.Box 27272, Ash Shariqah,  UAE }

\address{$^3$ Departamento de Matemática, Universidad de Tarapaca, Arica, Chile }

\email{ibtissam.issa@uniba.it, asoufyane@sharjah.ac.ae, opverav@academicos.uta.cl }


\begin{abstract}
In this paper, we examine the stability of a heat-conducting copper rod and a magnetizable piezoelectric beam, with fractional damping affecting the longitudinal displacement of the piezoelectric material's centerline. We establish a polynomial stability result that is dependent on the order of the fractional derivative.
\end{abstract}

\maketitle
\pagenumbering{roman}
\maketitle
\pagenumbering{arabic}
\setcounter{page}{1}
\section{Introduction}
Piezoelectric materials have emerged as a prominent mechanism for vibrational-to-electrical energy transduction, alongside electromagnetic and electrostatic methods  \cite{Williams1996}. Over the last decade, piezoelectric energy harvesting has gained significant attention in numerous review articles due to its high power densities and ease of application \cite{Adhikari2009}.\\

Piezoelectric materials serve as a sustainable source of electricity \cite{article2017}, extracted from various renewable sources such as oceanic wave energy \cite{Wu2015} and wind energy \cite{Oh2010}. These materials find implementation in various devices like microphones, load cells, and power excavators, showcasing excellent properties such as greater energy density, natural inverse energy conversion ability, and simplicity in design and construction.\\

Since the nineteenth century, materials like quartz, Rochelle salt, barium titanate, and Polyvinylidene Fluoride (PVDF) have been known to exhibit the direct piezoelectric effect, generating electric charge/voltage under pressure. The Curie brothers discovered this phenomena in 1880. Gabriel Lippmann discovered the converse piezoelectric effect in 1881, which is observable when these materials undergo an electric field and produce proportional geometric tension. Piezoelectric materials with multifunctional features include lead titanate, barium titanate, and lead zirconate titanate (PZT). According to \cite{Smith2005}, these materials have the ability to produce electric displacement that is directly proportionate to the applied mechanical stress. The driving frequency affects actuators made of these materials, such as those that need electrical input \cite{ozer2023maximal}, which controls how quickly a piezoelectric beam vibrates or changes states. These materials are used in energy harvesters, sensors, actuators, and other devices \cite{Baur2014}.\\
Because of its special ability to transform mechanical energy into electrical and magnetic energy and vice versa, piezoelectric actuators are incredibly versatile and can be used as both actuators and sensors. They compete with traditional actuators for many industrial jobs, especially those involving structural control, because they are generally more efficient, smaller, less costly, more scalable than traditional actuators. Applications for piezoelectric materials can be found in aerospace, automotive, industrial, and civil structures.\\

In the modeling of piezoelectric systems, three major effects—mechanical, electrical, and magnetic—need consideration, and their interrelations must be carefully addressed. Mechanical effects are typically modeled through Kirchhoff, Euler-Bernoulli, or Mindlin-Timoshenko small displacement assumptions (see, for instance, \cite{Banks1996,Smith2005, Yang2004AnIT} and references therein). To include electrical and magnetic effects, three main approaches, electrostatic, quasi-static, and fully dynamic are employed \cite{Tiersten1969}. While electrostatic and quasi-static approaches are widely used (see, for instance, \cite{Destuynder1992, Kapitonov2007, LASIECKA2009167, Smith2005}), these models completely exclude magnetic effects and their coupling with electrical and mechanical effects. In the electrostatic approach, electrical effects are stationary despite dynamic mechanical equations, and in the quasi-static approach, magnetic effects are still ignored, but electric charges exhibit time dependence. The electromechanical coupling is not dynamic.\\

A piezoelectric beam, characterized as an elastic beam with electrodes at its top and bottom surfaces, insulated edges (to prevent fringing effects), and connection to an external electric circuit, serves as a fundamental structure for studying the interaction between electrical and mechanical energy. Experimental observations suggest that magnetic effects are minor in the overall dynamics for polarized ceramics (see \cite{Yang2006}), leading to their exclusion in piezoelectric beam models. When subjected to a voltage source, a single piezoelectric beam either shrinks or extends.\\
For a beam of length $L$ and thickness $h$, classical models with an undamped Euler-Bernoulli beam model and the electrostatic or quasi-static assumptions describe the stretching motion as
\begin{equation}\label{Sys1}
\begin{array}{l}
\rho v_{tt}-\alpha_1 v_{xx}=0\quad (x,t)\in (0,L)\times \mathbb{R}^{+}\\
v(0,t)=0, \alpha_1v_x(L,t)=-\frac{\gamma V(t)}{h}, \quad t\in \mathbb{R}^{+}
\end{array}
\end{equation}
where $\rho, \alpha_1, \gamma$ denote mass density, elastic stiffness, and piezoelectric coefficients of the beam, respectively, $V(t)$ denotes the voltage applied at the electrodes, and $v$ denotes the longitudinal displacement of the beam. The control system \eqref{Sys1} is a well-posed boundary control system on a Hilbert space, with norm corresponding to the system energy. It has been shown that it is exactly controllable and hence exponentially stabilizable. In fact, by choosing the feedback $V (t) = -v_t (L, t)$, the solutions of the closed-loop system become exponentially stable \cite{Komornik1994}.\\
The authors in  \cite{Ozer2014,Ozer2013},  proposed a piezoelectric beam model with a magnetic effect, based on the Euler-Bernoulli and Rayleigh beam theory for small displacement; they considered an elastic beam covered by a piezoelectric material on its upper and lower surfaces, isolated by the edges and connected to a external electrical circuit to feed charge to the electrodes. As the voltage is prescribed at the electrodes, the following Lagrangian was considered:
\begin{equation}
\mathcal{L}=\int_0^T(\mathbf{K}-(\mathbf{P}+\mathbf{E})+\mathbf{B}+\mathbf{W})dt
\end{equation}
where $\mathbf{K}, \mathbf{P+E}, \mathbf{B}$ and $\mathbf{W}$ represent the (mechanical) kinetic energy, total stored energy, magnetic energy (electrical kinetic) of the beam and the work done by external forces, respectively.  Considering $v=v(x,t)$, $w=w(x,t)$ and $p=p(x,t)$ as functions that represent the longitudinal displacement of the center line, transverse displacement of the beam and the total load of the electric displacement along the transverse direction at each point $x$, respectively.  So, one can assume that
$$\mathbf{P+E}=\frac{h}{2}\int_0^L\left[\alpha\left(v_x^2+\frac{h^2}{12}w_{xx}^2-2\gamma\beta v_xp_x+\beta p_x^2\right)\right]dx,\quad\mathbf{B}=\frac{\mu h}{2}\int_0^Lp_t^2dx
$$
$$\mathbf{K}=\frac{\rho h}{2}\int_0^L\left[v_t^2+\left(\frac{h^2}{12}+1\right)w_t^2\right]\quad, \mathbf{W}=-\int_0^Lp_xV(t)dx.$$
From Hamilton’s principle for admissible displacement variations ${v, w, p}$ of $\mathcal{L}$ to zero and observing that the only external force acting on the beam is the voltage at the electrodes (the bending equation is decoupled) , they got the system
\begin{equation}\label{Initial model}
\begin{array}{ll}
\rho v_{tt}-\alpha v_{xx}+\gamma\beta p_{xx}=0,\\
\mu p_{tt}-\beta p_{xx}+\gamma \beta v_{xx}=0,
\end{array}
\end{equation}
where $\rho, \alpha,\gamma,\mu$ and $\beta$  denote the mass density, elastic stiffness, piezoelectric coefficient, magnetic permeability, water resistance coefficient of the beam and the prescribed voltage on electrodes of beam, respectively, and in addition, the relationship $\alpha=\alpha_1+\gamma^2\beta$.\\
They assumed that the beam is fixed at $x = 0$ and free at $x = L$, and thus, they got (from modelling) the following boundary conditions
\begin{equation}\label{BC}
\begin{array}{ll}
v(0,t)=p(0,t)=\alpha v_x(L,t)-\gamma\beta p_x(L,t)=0,\\
\beta p_x(L,t)-\gamma\beta v_x(L,t)=-\frac{V(t)}{h}.
\end{array}
\end{equation}
Then, the authors considered $V(t) = \kappa p_t(L,t)$ (electrical feedback controller) in \eqref{BC} and established strong stabilization for almost all system parameters and exponential stability for system parameters in a null measure set. In \cite{Ramos2018}, the authors,  inserted a dissipative term $\delta v_t$ in the first equation of \eqref{Initial model}, where $\alpha > 0$ is a constant and considered the following boundary condition
\begin{equation}
\begin{array}{ll}
v(0,t)=\alpha v_x(L,t)-\gamma\beta p_x(L,t)=0,\\
p(0,t)=\beta p_x(L,t)-\gamma\beta v_x(L,t)=0.
\end{array}
\end{equation}
The authors demonstrated the exponential decay of the system's energy using the energy technique. This indicates that the system is uniformly stabilized through the combined action of the friction term and the magnetic effect. A one-dimensional dissipative system of piezoelectric beams with a magnetic effect and localized damping was examined by the authors in \cite{Afilal2023}. Using a damping mechanism acting on only one component and a small portion of the beam, they demonstrated the exponential stability of the system. \\

In recent studies, researchers have explored the Timoshenko system, hybrid systems, coupled wave equations and Euler-Bernoulli beams incorporating control mechanisms based on fractional derivatives see \cite{AkilChitour2020,AkilMCRFIssa,AkilWehbe2019MCRF, Benaissa2019,Maryati2019, AkilGhader2020}. The consideration of such a damping is crucial in both theoretical and practical concepts and it describes memory and hereditary properties in various materials  \cite{app}. For instance, in viscoelasticity, materials like soils, concrete, rubber, biological tissue, and polymers exhibit elastic solid and viscous fluid-like responses (see \cite{Ronald,peter,RL,Bonetti}). Fractional calculus extends the traditional derivative to real orders, with various representations like Hadamard, Erdelyi-Kober, Riemann-Liouville, Riesz, Weyl, Gr\"unwald-Letnikov, Jumarie, and Caputo. Analyzing fractional dynamical systems is crucial for defining the fractional derivative appropriately. The Riemann-Liouville definition has physically unacceptable initial conditions, while the Caputo representation, introduced by Michele Caputo in 1967, uses integer-order derivatives with direct physical significance, mainly for memory effects. A recent alternative presented in \cite{Caputo-Fab} offers a fractional derivative without a singular kernel, providing unique properties like describing fluctuations and structures with different scales.\\

In \cite{Soufyane2021},  the authors considered a one-dimensional piezoelectric beams with magnetic effect damped with a weakly nonlinear feedback in the presence of a nonlinear delay term. They established an energy decay rate under appropriate assumptions on the weight of the delay. In \cite{Akil2022}, the author  investigated the stabilization of a system of piezoelectric beams under (Coleman or Pipkin)–Gurtin
thermal law with magnetic effect. He proved exponential
stability result for the  piezoelectric Coleman–Gurtin system and a polynomial energy decay rate of type $t^{-1}$ for the piezoelectric Gurtin–Pipkin system. 
In \cite{an2022stability}, the authors studied the stability of a piezoelectric beams with magnetic effects of fractional derivative type and with/ without thermal effects of Fourier’s law; they obtained an exponential stability by taking two boundary fractional dampings and additional thermal effect. In \cite{JDEOzer2021}, the authors utilized a variational approach to model vibrations on a piezoelectric beam with fractional damping, considering magnetic and thermal effects through Maxwell's equations and Fourier law. They proved the existence and uniqueness of solutions using semigroup theory, revealing smooth global attractors and exponential attractors. In \cite{AkilSoufyane2023}, the authors considered the stabilization of a one-dimensional piezoelectric with partial viscous dampings, they proved that it is sufficient to control the stretching of the center-line of the beam in x-direction to achieve the exponential stability along with numerical results to validate the theoretical result.  Recently in \cite{AkilOzer2024},the stability of longitudinal vibrations in transmission problems for two smart-system designs was investigated. The first design is a serially-connected Elastic-Piezoelectric-Elastic configuration with local damping applied solely to the piezoelectric layer. The authors demonstrated exponential stability for this system. The second design is a serially-connected Piezoelectric-Elastic configuration with local damping applied only to the elastic part. They showed that the stability of this system can be either polynomial or exponential, depending entirely on the arithmetic nature of a quotient involving all physical parameters.
\\

An important application of piezoelectric beams involves the interaction between a piezoelectric beam and a material that transfers heat. This interaction occurs at the joint of two beams, and is governed by transmission conditions \cite{Baur2014}. The fundamental PDE model for this physics problem can also be the same for the linearized version of the one-dimensional fluid equations by Navier-Stokes (heat equation)-structure (piezoelectric beam) interactions. This relationship is discussed in [\cite{Lions1985}, Chapter 7],\cite{zhang_zuazua_2007}, and is essential in various scenarios such as airflow along the aircraft \cite{Castille} or piezoelectric energy harvesting under the ocean.  Recently, in \cite{ozer2023exponential}, the authors examines a system involving a heat-transferring copper rod and a magnetizable piezoelectric beam in a transmission line setting. The heat and beam interactions are not exponentially stable, so two types of boundary-type state feedback controllers are proposed: static and hybrid. The PDE system is shown to have exponentially stable solutions using Lyapunov functions and various multipliers. 

To the best of our best knowledge,  the heat-transferring copper rod and a magnetizable piezoelectric beam with fractional damping acting on the longitudinal displacement of the center line of the piezoelectric material is not treated mathematically in the literature.  In pursuit of this objective, this paper endeavors to investigate the stability of such a system, wherein the copper rod and the magnetizable piezoelectric beam are interconnected via transmission conditions. The system is given as follows:

\begin{eqnarray}
\left\lbrace
\label{102}
\begin{array}{l}
z_{t} - \kappa z_{xx} = 0,\quad x\in (-\ell_{1},\,0),\quad t \geq 0, \\
\rho v_{tt} - \chi v_{xx} + \gamma\beta p_{xx} + \partial_{t}^{\alpha,\,\eta}v = 0,\quad x\in (0,\,\ell_{2}), \quad t \geq 0,  \\
\mu p_{tt} - \beta p_{xx} + \gamma\beta v_{xx}  = 0,\quad x\in (0,\,\ell_{2}),\quad t \geq 0, \\
z(-\ell_{1},\,t)  = 0, \quad t \geq 0,\\ 
z(0,\,t)= v_t(0,\,t), \quad t \geq 0,\\ 
p(\ell_{2},\,t)  = v(\ell_{2},\,t) = 0, \quad t \geq 0,\\ 
\chi v_x(0,\,t)-\gamma\beta p_x(0,\,t)=\kappa z_x(0,\,t), \quad t \geq 0,\\
\beta p_x(0,\,t)=\gamma\beta v_x(0,\,t), \quad t \geq 0,
\end{array}
\right. 
\end{eqnarray}
with the following initial conditions
\begin{equation}
\left\{\begin{array}{ll}
z(x,\,0) = z_{0}(x),&\quad x\in (-\ell_{1},\,0), \\
v(x,\,0) =  v_{0}(x),\quad v_{t}(x,\,0) = v_{1}(x),&\quad x\in (0,\,\ell_{2}), \\
p(x,\,0) =  p_{0}(x),\quad p_{t}(x,\,0) = p_{1}(x),&\quad x\in (0,\,\ell_{2}),
\end{array}\right.
\end{equation}
where $z = z(x,\,t)\,$ denote the heat distribution on the copper rod,  $v = v(x,\,t)$ denote the longitudinal oscillations of the center line of the beam and $p = p(x,\,t)$ the total charge accumulated at the electrodes of the beam.  The parameters $\chi,$ $\chi_{1},$ $\gamma,$ $\beta,$ $\rho,$ and $\mu$ represent the piezoelectric material parameters. Each parameter is listed in the table below, along with its corresponding significance. Moreover, the stiffness coefficient $\chi,$ due to the involvement of the piezoelectric, is different from the one on the fully elastic material $\chi_{1}$ and $\chi_1=\chi-\gamma^2\beta>0$. 
\begin{table}[h!]
\centering
\begin{tabular}{||c |c ||} 
 \hline
  Parameter & Corresponding Significance \\ [0.7ex] 
 \hline\hline
$\chi$ & The elastic stiffness coefficient  \\ [0.08in]
$\kappa$ & The thermal diffusivity constant\\[0.08in]
$\gamma$ & Piezoelectric coefficient  \\[0.08in]
 $\beta$ & The beam coefficient of impermeability  \\[0.08in]
$\rho$ & Mass density per unit volume,  \\[0.08in]
$\mu$& Magnetic permeability of beam\\ [1ex] 
 \hline
\end{tabular}
\label{table:1}
\caption{Table}
\end{table}

\renewcommand{\theequation}{\thesection.\arabic{equation}}
\setcounter{equation}{0}
\section{Augmented Model and Preliminaries}

\label{}

\noindent
As mentioned in the introduction, fractional calculus is widely used. Here, we provide a brief review of fractional calculus. Unlike for the fractional integral, there are several slightly different definitions for the fractional derivative operator. For a thorough understanding of the concept in the Caputo sense the reader can refer to  \cite{C1, C2, C3, KST}.
\\
\noindent
Let $0 < \alpha < 1$.  The Caputo fractional integral of order
$\alpha$ is defined by
\begin{align}\nonumber
	I^{\alpha}f(t) =
\frac{1}{\Gamma(\alpha)}\int_{0}^{t}(t - s)^{\alpha - 1}
f(s)ds,
\end{align}
where $\Gamma$ is the well-known gamma function, $f \in
L^{1}(0,\,+\infty)$. 
The Caputo fractional derivative operator of order $\alpha$ is
defined by
\begin{equation}\nonumber 
	D^{\alpha}f(t) = I^{1 - \alpha}D f(t) :=
\frac{1}{\Gamma(1 - \alpha)}\int_{0}^{t}(t - s)^{-\alpha}
f'(s)ds,
\end{equation}
with $f \in W^{1,\,1}(0,\,+\infty).$
\noindent 
We note that Caputo definition of fractional derivative does possess a
very simple but interesting  interpretation: if the function $f(t)$
represents the strain history within a viscoelastic material whose
relaxation function is $[\Gamma(1 - \alpha)t^{\alpha}]^{-1}$ then
the material will experience at  any time $t$ a total stress given
the expression $D^{\alpha}f(t).$
Moreover, it is easy to show that  $D^{\alpha}$ is a left inverse of
$I^{\alpha},$  but in general it is not a right inverse. More
precisely, we have
\begin{eqnarray*}
D^{\alpha}I^{\alpha}f =  f, \qquad I^{\alpha}D^{\alpha}f(t) =  f(t)
- f(0).
\end{eqnarray*}
For the proof of above equalities and more properties of fractional
calculus  see \cite{SKM}. \\
\\
In this work, we consider slightly different versions in 
\eqref{201} and \eqref{202}.  Indeed, Choi and
MacCamy \cite{CM} establish the following definition of fractional
integro-differential operators  with weight exponential. Let  $0 <
\alpha < 1\,$ and $\eta \ge 0.$  The exponential fractional integral of
order $\alpha$ is defined by
\begin{equation}
\label{201} I^{\alpha,\,\eta}f(t) =
\frac{1}{\Gamma(\alpha)}\int_{0}^{t}e^{-\eta(t - s)}(t -
s)^{\alpha - 1}f(s)ds,
\end{equation}
with  $f \in L^{1}([0,\,+\infty)).$
\\
The exponential fractional derivative operator of order $\alpha$ is
defined by
\begin{equation}\label{202}
 \partial_{t}^{\alpha,\,\eta}f(t) =  \frac{1}{\Gamma(1 - \alpha)}
\int_{0}^{t}e^{-\eta(t - s)}(t -
s)^{-\alpha}f'(s)ds,
\end{equation}
with $f \in W^{1,\,1}([0,\,+\infty)).$
\begin{Remark}
Note that $\partial_{t}^{\alpha,\,\eta}f(t) = I^{1 - \alpha,\,\eta} f'(t).$
\end{Remark}
\noindent
The following results are going to be used later in the paper. First, we recall theorem 2 stated in \cite{15}.
\begin{Theorem}\label{theorem1}
Let $\alpha\in (0,1)$, $\eta\geq 0$ and
$$\mu(\xi)=\abs{\xi}^{\frac{2\alpha-1}{2}}$$ be the function defined almost everywhere on $\mathbb{R}$. Then the relation between the {\it Input} $U$ and the {\it Output} $O$ of the following system
\begin{eqnarray}
\partial_t\varphi(x,t, \xi)+(\xi^2+\eta)\varphi(x,t, \xi)-U(x,t)|\xi|^{\frac{2\alpha-1}{2}}&=&0,(x,\xi,t)\in (0,L)\times \mathbb{R}\times (0,\infty),\label{aug1}\\  
\varphi(x,0, \xi)&=&0,(x,\xi)\in (0,L)\times \mathbb{R},\label{aug2}\\
\displaystyle{O(x,t)-\mathfrak{C}\int_{\mathbb{R}}|\xi|^{\frac{2\alpha-1}{2}}\varphi(x,t, \xi)d\xi }&=&0,(x,t)\in (0,L)\times (0,\infty),\label{aug3}
\end{eqnarray}
is given by 
\begin{equation}\label{relation}
O=I^{1-\alpha,\eta}U,
\end{equation}
where 
$
\mathfrak{C}=\pi^{-1}\sin(\alpha \pi).
$
\end{Theorem}
\noindent In the preceding theorem, when utilizing the input $U(x,t)=v_{t}(x,t)$ and applying Equation \eqref{202}, the resulting output $O$ is expressed as:
$$
O(x,t)=I^{1-\alpha,\eta}v_{t}(x,t)=\frac{1}{\Gamma(1-\alpha)}\int_0^t(t-s)^{-\alpha}e^{-\eta(t-s)}\partial_sv(x,s)ds=\partial_t^{\alpha,\eta}v(x,t).
$$
Thus, by taking the input $U(x,t)=v_{t}(x,t)$ in Theorem \ref{theorem1} and using the above equation, we get 
\begin{equation}\label{augnew}
\begin{array}{llll}
\partial_t\varphi(x,t, \xi)+(\xi^2+\eta)\varphi(x,t, \xi)-v_{t}(x,t)\mu(\xi)&=&0,&(x,\xi,t)\in (0,L)\times \mathbb{R}\times (0,\infty),\\[0.1in]
\varphi(x,0, \xi)&=&0,&(x,\xi)\in (0,L)\times \mathbb{R},\\[0.1in]
\displaystyle{\partial_t^{\alpha,\eta}v(x,t)-\mathfrak{C}\int_{\mathbb{R}}\mu(\xi)\varphi(x,t, \xi)d\xi }&=&0,&(x,t)\in (0,L)\times (0,\infty).
\end{array}
\end{equation}
From system \eqref{augnew}, we deduce that system \eqref{102} can be recasted into the following augmented model 
\begin{align}
\label{209}
\begin{cases}
z_{t} - \kappa z_{xx} = 0, &\quad x\in (-\ell_{1},\,0),\quad t \geq 0, \\
\rho v_{tt}(x,\,t) - \chi v_{xx}(x,\,t) + \gamma\beta p_{xx}(x,\,t) + \mathfrak{C}\displaystyle\int_{\mathbb{R}}\mu(\xi)\varphi(x,\,t,\,\xi)
d\xi = 0, &\quad x\in (0, \ell_2),\quad t \geq 0, \\
\medskip
\mu p_{tt}(x,\,t) - \beta p_{xx}(x,\,t) + \gamma\beta v_{xx}(x,\,t) = 0,&\quad x\in (0, \ell_2),\quad t \geq 0, \\
\medskip
\varphi_{t}(x,\,t,\,\xi) + (|\xi|^{2} + \eta)\varphi(x,\,t,\,\xi)  = \mu(\xi)v_{t}(x,\,t), &\quad x\in (0, \ell_2),\quad t \geq 0,\xi\in \R, \\
z(-\ell_{1},\,t)  = 0,p(\ell_{2},\,t)  = v(\ell_{2},\,t) = 0,&\quad t \geq 0,\\ 
\medskip
z(0,\,t)= v_t(0,\,t),&\quad t \geq 0,\\ 
\medskip
\chi v_x(0,\,t)-\gamma\beta p_x(0,\,t)=\kappa z_x(0,\,t),&\quad t \geq 0,\\ 
\medskip
\beta p_x(0,\,t)=\gamma\beta v_x(0,\,t),&\quad t \geq 0,\\
\end{cases}
\end{align}
with the following initial conditions
\begin{equation}
\left\{\begin{array}{ll}
z(x,\,0) = z_{0}(x),&\quad x\in (-\ell_{1},\,0), \\
v(x,\,0) =  v_{0}(x),\quad v_{t}(x,\,0) = v_{1}(x),&\quad x\in (0,\,\ell_{2}), \\
p(x,\,0) =  p_{0}(x),\quad p_{t}(x,\,0) = p_{1}(x),&\quad x\in (0,\,\ell_{2}),\\
\varphi(x,\,0,\,\xi) = 0, &\quad x\in (0,\,\ell_{2}), \xi\in \R.
\end{array}\right.
\end{equation}
where $\ 0 < \alpha < 1,\ \eta \geq 0,$ $z=z(x,\,t),$ $v=v(x,\,t),$ $p=p(x,\,t)$ are real-valued functions and $(x,\,t,\,\xi)\in (-\ell_{1},\,\ell_{2})\times (0,\,+\infty)\times \mathbb{R}.$

\begin{Lemma}
\label{lemma201}
For every solution $(z,\,v,\,p)$ of the system
\eqref{209}, The total  energy $\mathcal{E}: \mathbb{R}^{+}\rightarrow \mathbb{R}^{+}$ is defined at time $t$ by
\begin{align}
\label{210}\frac{d}{dt}\mathcal{E}(t) = & -\kappa\int_{-l_1}^0 |z_x|^2dx  -\mathfrak{C}\int_0^{\ell_2} \int_{\mathbb{R}}(|\xi|^{2} + \eta)|\varphi(x,t,\,\xi)|^{2}d\xi,
\end{align}
with
\begin{eqnarray}
\label{211}\mathcal{E}(t) & = & \frac{1}{2}\left[{\rm TE} + {\rm MechKE} + {\rm MagKE} + {\rm PE} + \mathfrak{C}\|\varphi\|_{L^{2}(\mathbb{R};\,(0,\,\ell_{2}))}^{2}\right],
\end{eqnarray}
where
\begin{align*}
&{\rm Thermal\ Energy} = TE = \|z(x,\,t)\|_{L^{2}(-\ell_{1},\,0)}^{2},\\ 
&{\rm Mechanical\ Kinetic\ Energy} = {\rm MechKE} =\rho \|v_{t}(x,\,t)\|_{L^{2}(0,\,\ell_{2})}^{2}, \\
&{\rm Magnetic\ Kinetic\ Energy} = {\rm MagKE} =\mu \|p_{t}(x,\,t)\|_{L^{2}(0,\,\ell_{2})}^{2}, \\
&{\rm Potential\ Energy} = {\rm PE} = \chi_1\|v_x\|^2_{L^{2}(0,\,\ell_{2})},\\
&{\rm Electromechanical \ Energy} = {\rm ElectroMechE} = \beta\|\gamma v_x-p_x\|^2_{L^{2}(0,\,\ell_{2})}.
\end{align*}

\end{Lemma}
\noindent
{\it Proof.}  
Multiplying in \eqref{209}$_{1}$ by $\overline{z}$,  integrating over $x\in (-\ell_{1},\,0)$ and taking the real part, we obtain
\begin{eqnarray}
\label{213}
\frac{1}{2}\dfrac{d}{dt}\int_{-\ell_{1}}^{0}|z(x,\,t)|^{2}dx + \kappa\int_{-\ell_{1}}^{0}|z_{x}(x,\,t)|^{2}dx-\kappa\Re\left(z_x(0,t)\overline{z}(0,t)\right) = 0.
\end{eqnarray}
Multiplying \eqref{209}$_{2}$ by $\overline{v_{t}}$,  integrating over $x\in (0,\,\ell_{2}),$ and taking the real part we obtain
\begin{eqnarray}\label{214}
& &\dfrac{1}{2}\dfrac{d}{dt}\left[\int_{0}^{\ell_{2}}\rho|v_{t}(x,\,t)|^{2}dx + \chi\int_{0}^{\ell_{2}}|v_{x}(x,\,t)|^{2}dx\right] -\Re\left( \gamma\beta\int_{0}^{\ell_{2}}p_{x}(x,\,t)\overline{v_{tx}}(x,\,t)dx\right)  \nonumber \\
& & +\ \Re\left( \mathfrak{C}\int_{0}^{\ell_{2}}\overline{v_{t}}(x,\,t)\int_{\mathbb{R}} \mu(\xi)\varphi(x,\,t,\,\xi)d\xi dx\right)+\chi\Re\left(v_x(0,t)\overline{v_t}(0,t)\right ) \nonumber \\
&&-\gamma\beta\Re\left(p_x(0,t)\overline{v_t}(0,t)\right) = 0.
\end{eqnarray}
Multiplying \eqref{209}$_{3}$ by $\overline{p_{t}}$ and using similar calculations as above, we obtain
\begin{eqnarray}
\label{215}
\dfrac{1}{2}\dfrac{d}{dt}\left[\int_{0}^{\ell_{2}}\mu|p_{t}(x,\,t)|^{2}dx + \beta\int_{0}^{\ell_{2}}|p_{x}(x,\,t)|^{2}dx\right] -\Re\left( \gamma\beta\int_{0}^{\ell_{2}}v_{x}(x,\,t)p_{tx}(x,\,t)dx\right) \nonumber
\\+\beta\Re\left(p_x(0,t)\overline{p_t}(0,t)\right)-\gamma\beta \Re\left(v_x(0,t)\overline{p_t(0,t)}\right)= 0.
\end{eqnarray}
Adding equations \eqref{213}, \eqref{214} and \eqref{215} and using $\dfrac{d}{dt}(v_{x}(x,\,t)p_{x}(x,\,t)) = v_{x}(x,\,t)p_{tx}(x,\,t) + v_{tx}(x,\,t)p_{x}(x,\,t),$  $\chi_1=\chi-\gamma^2 \beta$ and the transmission conditions lead to
\begin{eqnarray}
\label{216}
& & \frac{1}{2}\dfrac{d}{dt}\left[\int_{-\ell_{1}}^{0}|z(x,\,t)|^{2}dx +\rho \int_{0}^{\ell_{2}}|v_{t}(x,\,t)|^{2}dx + \chi_1\int_{0}^{\ell_{2}}|v_{x}(x,\,t)|^{2}dx \right. \nonumber  \\
& & \left.  \qquad\quad +\mu \int_{0}^{\ell_{2}}|p_{t}(x,\,t)|^{2}dx + \beta\int_{0}^{\ell_{2}}|\gamma v_x(x,\,t)-p_{x}(x,\,t)|^{2}dx \right] \nonumber \\
& &   \qquad\quad +\kappa \int_{-\ell_{1}}^{0}|z_{x}(x,\,t)|^{2}dx +\Re\left( \mathfrak{C}\int_{0}^{\ell_{2}}\overline{ v_{t}}(x,\,t)\int_{\mathbb{R}}\mu(\xi)\varphi(x,\,t,\xi)d\xi dx\right) = 0.
\end{eqnarray}
Finally, multiplying \eqref{209}$_{4}$ by $\mathfrak{C}\overline{\varphi}$ and integrating over $\xi\in\mathbb{R}$ and $x\in L^{2}(0,\,\ell_{2})$, and taking the real part,  we get
\begin{eqnarray}
\label{217}
& & \dfrac{\mathfrak{C}}{2}\dfrac{d}{dt}\int_{0}^{\ell_{2}}\int_{\mathbb{R}}|\varphi(x,\,t,\,\xi)|^{2}d\xi dx + \Re\left(\int_{0}^{\ell_{2}}\mathfrak{C}\int_{\mathbb{R}} (|\xi|^{2} + \eta)|\varphi(x,\,t,\,\xi)|^{2}d\xi dx\right)\nonumber \\
& = & \int_{0}^{\ell_{2}}\mathfrak{C}\overline{v_{t}}(x,\,t)\int_{\mathbb{R}}\mu(\xi)\varphi(x,\,t,\,\xi)d\xi dx.
\end{eqnarray}
Combining \eqref{216} and \eqref{217} yields
\begin{eqnarray*}
& &\dfrac{1}{2}\dfrac{d}{dt}\left[\|z(x,\,t)\|_{L^{2}(-\ell_{1},\,0)}^{2} + \rho\|v_{t}(x,\,t)\|_{L^{2}(0,\,\ell_{2})}^{2} +\mu  \|p_{t}(x,\,t)\|_{L^{2}(0,\,\ell_{2})}^{2} \right. \\
& & \left.\qquad\quad +\chi_1\|v_{x}(x,\,t)\|_{L^{2}(0,\,\ell_{2})}^{2} +\beta \|\gamma v_x(x,\,t)-p_{x}(x,\,t)\|_{L^{2}(0,\,\ell_{2})}^{2} + \mathfrak{C}\|\varphi(t,\,\xi)\|_{L^{2}(\R, (0,\,\ell_{2}))}^{2}d\xi\right] \\
& &\qquad = -\kappa  \|z_{x}(x,\,t)\|_{L^{2}(-\ell_{1},\,0)}^{2} - \mathfrak{C}\int_{\mathbb{R}} (|\xi|^{2} + \eta)\|\varphi(t,\,\xi)\|_{L^{2}(0,\,\ell_{2})}^{2}d\xi. \nonumber
\end{eqnarray*}

\renewcommand{\theequation}{\thesection.\arabic{equation}}
\setcounter{equation}{0}
\section{Setting of the Semigroup}
\label{}
\noindent In this section, we use results of the semigroup theory of linear operators   to obtain an existence theorem for the solutions of the system
\eqref{209} (see \cite{Pazy}). First we define the following Hilbert spaces

\medskip
 
\noindent  
$$H_R^1(0,\ell_2)=\left\{u\in H^1(0,\ell_2); u(l_2)=0\right\}$$
$$ \text{and} \quad
 H_L^1(-\ell_1,0)=\left\{u\in H^1(-\ell_1,0); u(-l_1)=0\right\}.$$
The energy space is defined by
 \begin{eqnarray}\label{EnergySpace}
\mathcal{H} =  L^{2}(-\ell_{1},\,0)\times H_{R}^{1}(0,\,\ell_{2})\times L^{2}(0,\,\ell_{2})\times H_{R}^{1}(0,\,\ell_{2}) \times L^{2}(0,\,\ell_{2}) \times L^{2}(\mathbb{R};\,L^{2}(0,\,\ell_{2})),
\end{eqnarray}
 and equipped with the inner product given by 
\begin{equation}
\begin{array}{ll}
\displaystyle
\langle \mathbb{U},\,\widetilde{\mathbb{U}}\rangle_{{\mathcal{ H}}} =  \int_{-\ell_1}^0|z|^2dx + \int_0^{\ell_2}\left(\chi_1|v_x|^2+\rho|\textsf{V}|^2
+ \mu|\textsf{P}|^2\right)+ \int_0^{\ell_2}\beta|\gamma v_x-p_x|^2dx
+\mathfrak{C} \int_0^{\ell_2}\int_{\R}|\varphi(x,\xi)|^2dx,
\end{array}
\end{equation}
where $\mathbb{U} = (z,\,v,\,\textsf{V},\,p,\,\textsf{P},\,\varphi)^{T}$,  $\widetilde{\mathbb{U}}
= (\tilde{z},\,\tilde{v},\,\tilde{\textsf{V}},\,\tilde{p},\,\tilde{\textsf{P}},\,\tilde{\varphi})^{T}\in \mathcal{H}.$ 
The associated norm is given by
\begin{eqnarray}\label{normh}
\|\mathbb{U}\|_{\mathcal {H}}^{2} = \|z\|_{L^{2}(-\ell_{1},\,0)}^{2} + \chi_1  \|v_x\|_{L^{2}(0,\,\ell_{2})}^{2}+\rho \|\textsf{V}\|_{L^{2}(0,\,\ell_{2})}^{2} +\mu \|\textsf{P}\|_{L^{2}(0,\,\ell_{2})}^{2} \nonumber  \\
+\beta \|\gamma v_x-p_x\|_{L^{2}(0,\,\ell_{2})}^{2} + \mathfrak{C}\|\varphi\|_{L^{2}(\mathbb{R}:\,L^{2}(0,\,\ell_{2}))}^{2}.
\end{eqnarray}
We note that the standard norm on $\mathcal{H}$ is given as follows
\begin{equation}\label{norms}
\begin{array}{lll}
\|\mathbb{U}\|_{\mathcal {S}}^{2} = \|z\|_{L^{2}(-\ell_{1},\,0)}^{2} &+ \|\textsf{V}\|_{L^{2}(0,\,\ell_{2})}^{2} + \|\textsf{P}\|_{L^{2}(0,\,\ell_{2})}^{2} + \|v_x\|_{L^{2}(0,\,\ell_{2})}^{2} \\
&+ \|p_x\|_{L^{2}(0,\,\ell_{2})}^{2} +\|\varphi\|_{L^{2}(\mathbb{R}:\,L^{2}(0,\,\ell_{2}))}^{2}.
\end{array}
\end{equation}
\begin{Lemma}\label{LemmaEquiv}
The norm defined in \eqref{normh} is equivalent to the standard norm in \eqref{norms} on $\mathcal{H}$, i.e. for all $\mathbb{U} = (z,\,v,\,\textsf{V},\,p,\,\textsf{P},\,\varphi)^{T}\in \mathcal{H}$,  there exist two positive constants $C_1$, $C_2$, independent of $\mathbb{U}$, such that
\begin{equation}\label{NormsEquiva}
C_1\|\mathbb{U}\|_{\mathcal{S}}\leq \|\mathbb{U}\|_{\mathcal{H}}\leq C_2 \|\mathbb{U}\|_{\mathcal{S}}.
\end{equation}
\end{Lemma}
\begin{proof}
Young's inequality yields
\begin{equation}
\beta\|\gamma v_x-p_x\|^2\leq  \max(2\beta \gamma^2,2\beta)\left(\|v_x\|^2_{L^{2}(0,\,\ell_{2})}+\|p_x\|^2_{L^{2}(0,\,\ell_{2})}\right).
\end{equation}
Then, the right hand side of \eqref{NormsEquiva} is derived with $C_2=\max\left(1, \rho, \mu, \mathfrak{C}, \chi_1+2\beta \max(\gamma^2,1)\right)$.
Next, we have that
\begin{equation*}
\|p_x\|^2_{L^{2}(0,\,\ell_{2})}\leq 2\|\gamma v_x-p_x\|^2_{L^{2}(0,\,\ell_{2})}+2\|v_x\|^2_{L^{2}(0,\,\ell_{2})}\leq 2\max\left(\frac{1}{\beta},\frac{\gamma^2}{\chi_1}\right)\left(\beta \|\gamma v_x-p_x\|^2_{L^{2}(0,\,\ell_{2})}+\chi_1\|v_x\|^2_{L^{2}(0,\,\ell_{2})} \right).
\end{equation*}
Thus, by utilizing the above inequality in \eqref{norms}, we obtain the left hand side of \eqref{NormsEquiva} with
$$C_1=\frac{1}{\max\left(1,\frac{1}{\rho}, \frac{1}{\mu},\frac{1}{\mathfrak{C}}, 2\max\left(\frac{1}{\beta},\frac{\gamma^2}{\chi_1}\right)\right)}.$$
\end{proof}

\noindent Now, we wish to transform the initial boundary value problem
\eqref{209} to an abstract problem in the Hilbert space
$\mathcal{H}.$ We introduce the functions $v_{t} = \textsf{V},\,$ $p_{t} = \textsf{P},$  and rewrite the system \eqref{209} as the following initial value
problem
\begin{eqnarray}
\label{403}\frac{d}{dt}\mathbb{U}(t) = \mathcal{\mathcal{ A}}\mathbb{U}(t),\quad \mathbb{U}(0) = \mathbb{U}_{0},\quad
\forall\, t > 0,
\end{eqnarray}
where $\mathbb{U}=(z,\,v,\,\textsf{V},\,p,\,\textsf{P},\,\varphi)^{T}\ $ and
$\ \mathbb{U}_{0}=(z_{0},\,v_{0},\,v_{1},\,p_{0},\,p_{1},\,0)^{T},
$ and the operator $\,{\mathcal{ A}}:\mathcal{D}({\mathcal{ A}})\subset
\mathcal{H}\rightarrow \mathcal{H}$ is given by
\begin{equation}
\label{404}\mathcal{\mathcal{ A}}\left(
\begin{array}{c}
z \\
\\
v \\
\\
\textsf{V} \\
\\
\textsf{p} \\
\\
\textsf{P} \\
\\
\varphi 
\end{array}
\right) = \left(
\begin{array}{c}
\kappa z_{xx} \\
\\
\textsf{V} \\ 
\\
\frac{1}{\rho}\left( \chi v_{xx} - \gamma\beta p_{xx} - \mathfrak{C}\displaystyle\int_{\mathbb{R}}\mu(\xi)\varphi(\xi)d\xi\right) \\
 \\
\textsf{P} \\
\\
\frac{1}{\mu}\left(\beta p_{xx} - \gamma\beta v_{xx}\right) \\
\\
-(|\xi|^{2} + \eta)\varphi(\xi) + \mu(\xi)\textsf{V}
\end{array}
\right)
\end{equation}
with the domain
\begin{equation*}
\mathcal{D}({\mathcal{ A}})=
\left\{
\begin{array}{ll}
(z,\,v,\,\textsf{V},\,p,\,\textsf{P},\,\varphi)^{T}\in {\mathcal{ H}}:\, z\in H_{L}^{1}(-\ell_{1},\,0)\cap H^{2}(-\ell_{1},\,0),\ 
v,\,p\in H_{R}^{1}(0,\,\ell_{2})\cap H^{2}(0,\,\ell_{2}), \\[0.1in]

\textsf{V},\,\textsf{P}\in H_{R}^{1}(0,\,\ell_{2}), \quad |\xi|\varphi\in L^{2}(\mathbb{R};\,L^{2}(0,\,\ell_{2})),\,-(|\xi|^{2} + \eta)\varphi + \mu(\xi)\textsf{V}\in
L^{2}\left(\mathbb{R};\,L^{2}(0,\,\ell_{2})\right)\\[0.1in]
z(0)=V(0), \, \chi v_x(0)-\gamma\beta p_x(0)=\kappa z_x(0), 
\beta p_x(0)=\gamma\beta v_x(0).
\end{array}
\right\}
\end{equation*}

\noindent Now, we consider the following technical lemmas. Lemma \ref{L2.3} will be used for well-posedness and Lemma \ref{L2.4} will be used in the proof of strong stability. The proofs of these two lemmas are similar to Lemmas 2.1 and 2.2 in  \cite{AkilMCRFIssa}, for this reason we omit their proofs.
\begin{Lemma}(see Lemma 2.1 in \cite{AkilMCRFIssa}).
\label{L2.3}
If $0 < \alpha < 1$ and $\eta\ge 0,$ then 
\begin{align*}	
C(\alpha,\,\eta):=\displaystyle\int_{\mathbb{R}}\dfrac{|\xi|^{2\alpha - 1}}{|\xi|^{2} + \eta + 1}d\xi < +\infty
\qquad {\rm and} \qquad D(\alpha,\,\eta):=\displaystyle\int_{\mathbb{R}} \dfrac{|\xi|^{2\alpha - 1}}{(|\xi|^2 + \eta + 1)^{2}}d\xi < +\infty.
	\end{align*}
\end{Lemma}
\noindent

\begin{Lemma}(see Lemma 2.2 in \cite{AkilMCRFIssa}). 
\label{L2.4}
Let $0<\alpha<1.$ If $\eta>0$ and $\lambda\in \mathbb{R},$ or if $\eta = 0$ and $\lambda>0,$ then 
\begin{align*}	
E(\lambda,\,\alpha,\,\eta):= \displaystyle\int_{\mathbb{R}}\dfrac{|\xi|^{2\alpha - 1}\,d\xi}{|\xi|^{2} + \eta + i\lambda} < +\infty. 
	\end{align*}
Furthermore, for $h\in L^{2}(\mathbb{R};\,L^{2}(0,\,L)),$ we have that 
\begin{align*}
H(x,\,\lambda,\,\alpha,\,\eta):= \displaystyle\int_{\mathbb{R}}\dfrac{|\xi|^{\frac{2\alpha - 1}{2}}h(x,\,\xi)\,d\xi}{|\xi|^{2} +\eta + i\lambda}\in L^2(0,\,L). 
\end{align*}
\end{Lemma}

\begin{Proposition}
\label{lu}
The operator ${\mathcal{ A}}$ is the infinitesimal generator of a
contraction semigroup $\{{\mathcal{ S}}_{\mathcal{\mathcal{ A}}}(t)\}_{t\geq 0}.$
\end{Proposition}
\proof  We will show that ${\mathcal{ A}}$ is a dissipative operator and that 
$(I - {\mathcal{ A}})$ is surjective. Then,  the result will
 follow using the well known Lumer-Phillips
Theorem (see \cite{Pazy}). We observe that if $\mathbb{U} = (z,\,v,\,\textsf{V},\,p,\,\textsf{P},\,\varphi)^{T}\in {\mathcal{D}}({\mathcal{ A}}),$ then
using \eqref{209} and \eqref{210} we get
\begin{align}
\label{405}Re\left\langle{\mathcal{ A}}\mathbb{U},\,\mathbb{U}\right\rangle_{\mathcal{H}} = - \kappa\int_{-\ell_{1}}^{0}|z_x|^{2} dx
- \mathfrak{C}\int_{\mathbb{R}}(|\xi|^{2} +
\eta)\|\varphi(x,\,\xi)\|_{L^{2}(0,\,\ell_{2})}^{2}d\xi \leq 0.
\end{align}
In fact, let $\mathbb{U} = (z,\,v,\,\textsf{V},\,p,\,\textsf{P},\,\varphi)\in \mathcal{D}({\mathcal{A}}),$ then
\begin{align*}	
\langle {\mathcal{A}}\mathbb{U},\,\mathbb{U}\rangle_{{\mathcal{H}}} 
= &\ - \kappa\int_{-\ell_{1}}^{0}z_{x}\overline{z}_{x}dx + \chi\int_{0}^{\ell_{2}}\textsf{V}_{x}\overline{v}_{x}dx -\chi  \int_{0}^{\ell_{2}}{v}_{x}\overline{\textsf{V}_{x}} dx +\gamma\beta \int_{0}^{\ell_{2}}p_x\overline{\textsf{V}_{x}}dx\\
&-\mathfrak{C} \int_{0}^{\ell_{2}}\overline{\textsf{V}}\int_{\R}\mu(\xi)\varphi(x,\xi)d\xi dx 
- \beta \int_{0}^{\ell_{2}}p_x\overline{\textsf{P}_{x}}dx+\gamma\beta \int_{0}^{\ell_{2}}v_x\overline{\textsf{P}_{x}}dx+\gamma^2\beta  \int_{0}^{\ell_{2}}\textsf{V}_{x} \overline{v_x}dx\\
&-\gamma\beta \int_{0}^{\ell_{2}}\textsf{V}_{x} \overline{p_x}dx-\gamma\beta \int_{0}^{\ell_{2}}\textsf{P}_{x} \overline{v_x}dx+\beta \int_{0}^{\ell_{2}}\textsf{P}_{x} \overline{p_x}dx\\
&+\mathfrak{C} \ \int_{0}^{\ell_{2}}\int_{\mathbb{R}}[-(|\xi|^{2} + \eta)\varphi(x, \xi) + \mu(\xi)\textsf{V}]\overline{\varphi}(x,\xi)d\xi dx.
\end{align*}
Hence,
\begin{align*}
&\langle{\mathcal{A}}\mathbb{U},\,\mathbb{U}\rangle_{{\mathcal{H}}} = - \kappa\int_{-\ell_1}^0 |z_x|^2dx + 2i\chi Im\int_{0}^{\ell_{2}}\textsf{V}_{x}\overline{v}_{x}dx + 2i\beta Im\int_{0}^{\ell_{2}}\textsf{P}_{x}\overline{p}_{x}dx  \\
& + 2i\gamma\beta Im\int_{0}^{\ell_{2}}\left(\textsf{V}_x\overline{p_x}+\textsf{P}_x\overline{v_x}\right) + 2i\mathfrak{C}Im\int_{0}^{\ell_{2}}\overline{\textsf{V}}\int_{\mathbb{R}}\mu(\xi)\varphi(x, \xi)d\xi dx \\
& - \mathfrak{C} \int_{0}^{\ell_{2}}\int_{\mathbb{R}}(|\xi|^{2} + \eta)|\varphi(x, \xi)|^{2}d\xi dx.
\end{align*}
Taking the real part
\begin{align*}
Re\langle {\mathcal{A}}\mathbb{U},\,\mathbb{U}\rangle_{{\mathcal{H}}} = - \kappa\int_{-\ell_1}^0 |z_x|^2dx - \mathfrak{C} \int_{0}^{\ell_{2}}\int_{\mathbb{R}}(|\xi|^{2} + \eta)|\varphi(x, \xi)|^{2}d\xi dx\leq 0,
\end{align*}
which implies that $\mathcal{A}$ is dissipative.  
\noindent
Next, to use the Lumer-Phillips Theorem, it is sufficient to show that $(I - {\mathcal{A}})\mathbb{U} = \mathbb{F}.$ Thus, given $\mathbb{F} = (f_{1},\,f_{2},\,f_{3},\,f_{4},\,f_{5},\,f_{6}(x,\xi))^{T}\in \mathcal{H}$, we need to show that there is some vector  $\mathbb{U}=(z,\,v,\,\textsf{V},\,p,\,\textsf{P},\,\varphi)^{T}\in {\mathcal{D}}({\mathcal{A}})$ such that  $(I - {\mathcal{A}})\mathbb{U} = \mathbb{F}.$ That is, 
\begin{align}
\label{406}
\begin{cases}
z - \kappa z_{xx} = f_{1}  \\
v - \textsf{V} = f_{2}, \\
\rho \textsf{V} - \chi v_{xx} + \gamma\beta p_{xx} + \mathfrak{C}\displaystyle\int_{\mathbb{R}}\mu(\xi)\varphi(\xi)d\xi = \rho f_{3}, \\	
p - \textsf{P} = f_{4}, \\
\mu \textsf{P}  - \chi p_{xx} + \gamma\beta v_{xx} =\mu  f_{5}    \\ 
\varphi(\xi) + (|\xi|^{2} + \eta)\varphi(\xi) - \mu(\xi)\textsf{V} = f_{6}(x,\xi). 
\end{cases}
\end{align}
From \eqref{406}$_{2}$ and \eqref{406}$_{4}$,  we get 
\begin{equation}
\label{407}
\begin{cases}
\textsf{V} = v - f_{2}, \\ 
\textsf{P} = p - f_{4}.
\end{cases}
\end{equation} 
On the other hand, from the last equation in \eqref{406}$_{6},$ we obtain 
\begin{align*}
\label{408}\varphi(x,\xi) = \dfrac{f_{6}(x, \xi)}{|\xi|^{2} + \eta + 1} - \dfrac{\mu(\xi)f_{2}}{|\xi|^{2} + \eta + 1} + \dfrac{\mu(\xi)v}{|\xi|^{2} + \eta + 1}. 
\end{align*}
Multiplying the above equation by $\mu(\xi)$ and using  Lemma \ref{L2.3},  we get 
\begin{eqnarray*}
\label{409}
\mathfrak{C}\int_{\mathbb{R}}\mu(\xi)\varphi(x,\xi)d\xi
& = & \mathfrak{C}\left[\int_{\mathbb{R}}\dfrac{\mu(\xi)f_{6}(x,\xi)}{|\xi|^{2} + \eta + 1}d\xi + C(\alpha,\,\eta)(v - f_{2})\right] \nonumber \\
& = & \mathfrak{C}C(\alpha,\,\eta)(v - f_{2}) + \mathfrak{C}\int_{\mathbb{R}}\dfrac{\mu(\xi)f_{6}(x,\xi)}{|\xi|^{2} + \eta + 1}d\xi.
\end{eqnarray*}
Inserting \eqref{407}$_{1}$ into \eqref{406}$_{3}$ and usingthe above equality, we have
\begin{equation}
\label{410}
\rho v - \rho f_{2} - \chi v_{xx} + \gamma\beta p_{xx} + \mathfrak{C}C(\alpha,\,\eta)(v - f_{2}) + \mathfrak{C}\int_{\mathbb{R}}\dfrac{\mu(\xi)f_{6}(x,\xi)d\xi}{|\xi|^{2} + \eta + 1} = \rho f_{3}.
\end{equation}
In a similar way, inserting \eqref{407}$_{2}$ into \eqref{406}$_{5}$,  we get
\begin{equation}
\label{411}
\mu p -\mu  f_{4} - \beta p_{xx} + \gamma\beta v_{xx} =\mu  f_{5}.
\end{equation}
Multiplying equations \eqref{406}$_{1}$ by $\overline{\Gamma}\in H_{L}^{1}(-\ell_{1},\,0)$ and \eqref{410}, \eqref{411} by $\overline{\Theta},\,\overline{\Upsilon}\in [H_{R}^{1}(0,\,\ell_{2})]^{2}$ respectively such that $\Gamma(0)=\Theta(0)$, integrating over $x$ and performing integration by parts, one obtains the following 
\begin{equation}
\label{415}
\mathcal{B}((z,\,v,\,p),\,(\Gamma,\,\Theta,\,\Upsilon)) = \mathcal{L}(\Gamma,\,\Theta,\,\Upsilon),
\end{equation}
where $\mathcal{B}:\,[H_{L}^{1}(-\ell_{1},\,0)\times H_{R}^{1}(0,\,\ell_{2})\times H_{R}^{1}(0,\,\ell_{2})]\times [H_{L}^{1}(-\ell_{1},\,0)\times H_{R}^{1}(0,\,\ell_{2})\times H_{R}^{1}(0,\,\ell_{2})]\rightarrow\mathbb{R}$ is the sesquilinear form defined by
\begin{align*}
& \mathcal{B}((z,\,v,\,p),\,(\Gamma,\,\Theta,\,\Upsilon)) \\		
= &\displaystyle \int_{-\ell_{1}}^{0}z\Gamma dx+\kappa \int_{-\ell_{1}}^{0}z_x\Gamma_xdx+\ C_{1}\int_{0}^{\ell_{2}}v\Theta dx + \chi\displaystyle\int_{0}^{\ell_{2}}v_{x}\Theta_{x}dx - \gamma\beta\int_{0}^{\ell_{2}}p_{x}\Theta_{x}dx + 	\displaystyle \mu\int_{0}^{\ell_{2}}p\Upsilon dx \\
&\displaystyle+ \beta\displaystyle\int_{0}^{\ell_{2}}p_{x}\Upsilon_{x}dx - \gamma\beta\displaystyle\int_{0}^{\ell_{2}}v_{x}\Upsilon_{x}dx,
\end{align*}
and $\mathcal{L}:\,[H_{L}^{1}(-\ell_{1},\,0)\times H_{R}^{1}(0,\,\ell_{2})\times H_{R}^{1}(0,\,\ell_{2})]\rightarrow\mathbb{R}$ is the linear form defined by
\begin{center}
$\mathcal{L}(\Gamma,\,\Theta,\,\Upsilon) = \displaystyle\int_{-\ell_{1}}^{0}f_{1}\Gamma dx + \int_{0}^{\ell_{2}}F_{1}\Theta dx + \displaystyle\int_{0}^{\ell_{2}}F_{2}\Upsilon dx - \mathfrak{C}\int_{0}^{\ell_{2}}\Theta\displaystyle\int_{\mathbb{R}}\dfrac{\mu(\xi)f_{6}(\xi)d\xi}{|\xi|^{2} + \eta + 1}dx.$
\end{center}
where $C_{1} = \rho  + \mathfrak{C}C(\alpha,\,\eta),$ \ $F_{1} = (\rho + \mathfrak{C}C(\alpha,\,\eta))f_{2} + \rho f_{3}$\ and\ $F_{2} =\mu( f_{4} + f_{5}).$ \\
It is easy to observe that $\mathcal{B}$ is a continuous and coercive sesquilinear form. On the other hand, 
\begin{center}
$\left|\mathfrak{C}\displaystyle\int_{0}^{\ell_2}\Theta\displaystyle\int_{\mathbb{R}}\dfrac{\mu(\xi)f_{6}(\xi)d\xi}{|\xi|^{2} + \eta + 1}dx\right| \le \mathfrak{C}\sqrt{D(\alpha,\,\eta)}\,\|\Theta\|_{L^2(0,\ell_2)}\|f_{6}\|_{L^{2}(\mathbb{R};\,L^{2}(0,\,\ell_{2}))}.$
\end{center}
It follows that $\mathcal{L}$ is a bounded linear functional. 
Then, using Lax-Milgram's theorem, we deduce the existence of a unique solution $(z,\,v,\,p)\in H_{L}^{1}(-\ell_{1},\,0)\times H_{R}^{1}(0,\,\ell_{2})\times H_{R}^{1}(0,\,\ell_{2})$ to the variational problem \eqref{415}. By elliptic regularity, it follows that $z\in H_{L}^{1}(-\ell_{1},\,0)\cap H^{2}(-\ell_{1},\,0)$ and $v,\,p\in H_{R}^{1}(0,\,\ell_{2})\cap H^{2}(0,\,\ell_{2}).$ Now, define $V$ and $P$ as  \eqref{407}, then we have $V,\,P\in H_{R}^{1}(0,\,\ell_{2}).$ Finally, $f_{6}\in L^{2}(\mathbb{R};\,L^{2}(0,\,\ell_{2})),$ then defining $\varphi(x, \xi)$ and using the respective expressions given in \eqref{408}, it is evident that  $|\xi|\varphi\in L^{2}(\mathbb{R};\,L^{2}(0,\,\ell_{2}))$ and $-(|\xi|^{2} + \eta)\varphi + \mu(\xi)V\in L^{2}(\mathbb{R};\,L^{2}(0,\,\ell_{2})).$ \\
Therefore using Proposition \ref{lu}, the system \eqref{209} is well-posed in the energy space ${\mathcal{H}}$ and we have the following theorem:
\begin{Theorem}(Existence and Uniqueness of solutions)
\label{exis}
If $\,(z_{0},\,v_{0},\,v_{1},\,p_{0},\,p_{1},\,0)\in {\mathcal{H}},$ the problem \eqref{209} admits a unique weak solution
\begin{equation*}
(z,\,v,\,v_{t},\,p,\,p_{t},\,\varphi) \in C\left([0,\,+\infty);\,{\mathcal{H}}\right),
\end{equation*}
and for $(z_{0},\,v_{0},\,v_{1},\,p_{0},\,p_{1},\,0)\in {\mathcal{ D}}({\mathcal{A}})$, the problem \eqref{209} admits a unique strong solution
\begin{equation*}
(z,\,v,\,v_{t},\,p,\,p_{t},\,\varphi)\in C\left([0,\,+\infty);\,{\mathcal{D}}({\mathcal{A}})\right)\cap C^{1}\left([0,\,+\infty);\,{\mathcal{H}}\right).
\end{equation*}
\end{Theorem}

\section{Strong Stability}\label{strong stabilty section }
\noindent The aim of this section is to prove the strong stability of \eqref{209}.  The main result of this section is the following theorem. First, we state Arendt-Batty theorem which will be the base for the strong stability result.
\begin{Theorem}(Arendt-Batty \cite{Arendt-Batty})
\label{Arendt}
	Let ${\mathcal{A}}$ be the generator of a C$_{0}$-semigroup $({\mathcal{ S}}(t))_{t\ge 0}$ in a reflexive Banach space ${\it X.}$ If the following conditions are satisfied: 
	\begin{enumerate}
		\item [(i)] ${\mathcal{A}}$ has no purely imaginary eigenvalues;
		\item [(ii)] $\sigma({\mathcal{A}})\cap i\mathbb{R}$ is countable.
	\end{enumerate}
	Then, $({\mathcal{ S}}(t))_{t\ge 0}$ is strongly stable.
\end{Theorem}

\begin{Lemma}
\label{p4.2}
If $\eta = 0,$ then the operator ${\mathcal{A}}$ is not invertible and consequently $0\in \sigma({\mathcal{A}}).$
\end{Lemma}
\noindent
{\it Proof.} Let $\mathbb{F}_{0}=\left(0,\,\sin(\pi x/L),\,0,\,0,\,0,\,0\right)\in \mathcal{H}$ and assume there exists $\mathbb{U}_{0}=(z_{0},\,v_{0},\,\textsf{V}_{0},\,p_{0},\,\textsf{P}_{0},\,\varphi_{0})\in \mathcal{D}({\mathcal{A}})$ such that ${\mathcal{A}}\mathbb{U}_{0} = \mathbb{F}_{0}.$ In this case, $\varphi_{0}(x, \xi) = |\xi|^{\frac{2\alpha - 5}{2}}\sin(\pi x/L).$ However, $\varphi_{0}\notin L^{2}(\mathbb{R};\,L^{2}(0,\,\ell_{2}))$ for $0 < \alpha < 1.$ \quad $\square$

\begin{Remark}\label{Remark1A}
For the case when $\eta>0$, by using the same computations as in Proposition \ref{lu}, we prove that $-\mathcal{A}$ is m-dissipative operator and consequently $0\in \rho(\mathcal{A})$.
\end{Remark}

\begin{Theorem}\label{Strong1}
The $C_0$-semigroup of contractions $(e^{t\mathcal{A}})_{t\geq 0}$ is strongly stable in $\mathcal{H}$, i.e., for all $U_0\in \mathcal{H}$, the solution of  \eqref{403} satisfies $\displaystyle\mathcal{E}(t)\xrightarrow[t\to\infty] {} 0$.
\end{Theorem}
\noindent According to Theorem of Arendt-Batty, to prove Theorem \ref{Strong1}, we need to prove that the operator $\mathcal{A}$ has no pure imaginary eigenvalues and $\sigma(\mathcal{A})\cap i\R$ is countable. The proof of Theorem \ref{Strong1} will be established based on the following proposition.
\begin{Proposition}\label{proposition2}
We have
\begin{eqnarray}
&&\text{If}\quad \eta>0,  \, \text{we have}\quad i\R\subseteq \rho(\mathcal{A}),\label{pro1}\\
&&\text{If}\quad \eta=0,  \, \text{we have}\quad i\R^{\ast}\subseteq \rho(\mathcal{A}).\label{pro1*}
\end{eqnarray}
\end{Proposition}
\noindent We will prove Proposition \ref{proposition2} by a contradiction argument. Note that, according to Lemma \ref{p4.2}, for $\eta=0$, it follows that $0\in \sigma(\mathcal{A})$, hence $0\notin \rho(\mathcal{A})$. Furthermore, as indicated in Remark \ref{Remark1A}, we have $0\in\rho(\mathcal{A})$. Thus, our focus narrows to demonstrating that $i\R^{\ast}\subseteq \rho(\mathcal{A})$. Now, suppose that  $i\R^{\ast}\not\subseteq \rho(\mathcal{A})$ , then there exists $\omega\in \R^{\ast}$ such that $i\omega\notin \rho(\mathcal{A})$.  According to Remark A.3 in \cite{Akil2022} and page 25 in \cite{LiuZheng01} , there exists $\lbrace \la_n, U^n=(z^n,\,v^n,\,\textsf{V}^n,\,p^n,\,\textsf{P}^n,\,\varphi^n)^{\top}\rbrace_{n\geq 1}\subset \R^{\ast}\times D(\mathcal{A})$, such that
\begin{equation}\label{pro2}
\la_n\to \omega\quad \text{as}\quad \, n\to\infty\quad\text{and}\quad |\la_n|<|\omega|,
\end{equation}
\begin{equation}\label{pro3}
\|U^n\|_{\mathcal{H}}=\|(z^n,\,v^n,\,\textsf{V}^n,\,p^n,\,\textsf{P}^n,\,\varphi^n)^{\top}\|_{\mathcal{H}}=1,
\end{equation}
and
\begin{equation}\label{pro4}
(i\la_n I-\mathcal{A})U^n=F_n:=(f^1_n, f^2_n, f^3_n, f^4_n, f^5_n, f^6_n(\cdot,\xi))\to 0\quad\text{in}\quad\mathcal{H}, \quad\text{as}\quad n\to \infty.
\end{equation}
Detailing \eqref{pro4}, we get
\
\begin{equation}\label{pro5}
\left\{
\begin{array}{ll}
i\lambda_n z^n - \kappa z^n_{xx} = f^1_n, \\[0.1in]
i\lambda _nv^n - \textsf{V}^n = f^2_n,  \\[0.1in]
i\lambda_n\rho\textsf{V}^n - \chi v^n_{xx} + \gamma \beta p^n_{xx} + \mathfrak{C} \displaystyle\int_{\mathbb{R}}\mu(\xi)\varphi^n(x,\xi)d\xi = \rho f^3_n,  \\[0.1in]
i\lambda_n p^n - \textsf{P}^n = f^4_n \\[0.1in]
i\lambda\mu\textsf{P}^n - \beta p^n_{xx} + \gamma \beta v^n_{xx} =\mu f^5_n, \\[0.1in]
(|\xi|^{2} + \eta + i\lambda_n)\varphi^n(x,\xi) - \mu(\xi)\textsf{V} ^n= f^6_n(\cdot,\xi),\quad \forall \,\xi\in \mathbb{R}.
\end{array}\right.
\end{equation}
Inserting \eqref{pro5}$_2$ and \eqref{pro5}$_4$ into \eqref{pro5}$_3$ and \eqref{pro5}$_5$ respectively, we obtain
\begin{equation}\label{E2}
\left\{
\begin{array}{ll}
\displaystyle
\la _n^2\rho v^n+\chi v^n_{xx}-\gamma\beta p^n_{xx}-\mathfrak{C}\int_{\mathbb{R}}\mu(\xi)\varphi^n(x,\xi)d\xi =F^1_n,\\[0,2in]
\displaystyle
\la _n^2\mu p^n+\beta p^n_{xx}-\gamma\beta v^n_{xx}=F^2_n,
\end{array}\right.
\end{equation}
where $F^1_n=-(\rho f^n_3+i\la_n\rho f_2^n)$ and $F_n^2=-(\mu f_5^n+i\la_n\mu f_4^n)$.\\
By using the fact that $\chi=\chi_1+\gamma^2\beta$  in \eqref{E2}$_1$, we obtain
\begin{equation*}
\la _n^2\rho v+\chi _1v^n_{xx}+\gamma\left(\gamma\beta
v^n_{xx}-\beta p^n_{xx}\right)-\mathfrak{C}\int_{\mathbb{R}}\mu(\xi)\varphi^n(x,\xi)d\xi =F^1_n.
\end{equation*}
Combining the above equation with \eqref{E2}$_2$ leads to
\begin{equation*}
\chi _1v^n_{xx}=-\la _n^2\rho v^n-\la _n^2\gamma \mu p^n+\mathfrak{C}\int_{\mathbb{R}}\mu(\xi)\varphi^n(x,\xi)d\xi +F^1_n+\gamma F^2_n .
\end{equation*}
Inserting the above equation into \eqref{E2}, we obtain

\begin{eqnarray}
\displaystyle
\la _n^2\rho v^n+\chi _1v^n_{xx}+\la _n^2\gamma \mu p^n-\mathfrak{C}\int_{\mathbb{R}}\mu(\xi)\varphi^n(x,\xi)d\xi =F^1_n+\gamma F^2_n ,\label{e1}\\
\displaystyle
\la _n^2 \mu\chi p^n+\chi_1\beta p^n_{xx}+\la _n^2\rho\gamma\beta v^n-\gamma\beta\mathfrak{C}\int_{\mathbb{R}}\mu(\xi)\varphi^n(x,\xi)d\xi =\gamma\beta F^1_n+\chi F^2_n.\label{e2}
\end{eqnarray}
We will proof condition $i\R^{\ast}\subseteq \rho(\mathcal{A})$ by finding a contradiction with \eqref{pro3} such as $\|U^n\|_{\mathcal{H}}\xrightarrow[n\to\infty] {} 0$.  The proof of this Proposition will rely on the forthcoming Lemmas.

\begin{Lemma}\label{lemI4}
Let $\alpha\in (0,1)$, $\eta\geq 0$ and $\la\in \R$, then
$$
\begin{array}{l}
\displaystyle
\mathtt{J}_{1}(\la,\eta,\alpha)=\int_{\R}\frac{\abs{\xi}^{\alpha+\frac{1}{2}}}{\left(\abs{\la}+\xi^2+\eta\right)^2}d\xi=c_1\left(\abs{\la}+\eta\right)^{\frac{\alpha}{2}-\frac{5}{4}},\quad
\mathtt{J}_{2}(\la,\eta)=\left(\int_{\R}\frac{1}{(\abs{\la}+\xi^2+\eta)^2}d\xi\right)^{\frac{1}{2}}=\sqrt{\frac{\pi}{2}}\frac{1}{(\abs{\la}+\eta)^{\frac{3}{4}}},\\[0.1in]
\displaystyle
\hspace{2cm}\mathtt{J}_{3}(\la,\eta)=\left(\int_{\R}\frac{\xi^2}{\left(\abs{\la}+\xi^2+\eta\right)^4}d\xi\right)^{\frac{1}{2}}=\frac{\sqrt{\pi}}{4}\frac{1}{(\abs{\la}+\eta)^{\frac{5}{4}}}
\end{array}
$$
where $\displaystyle c_1=\int_{1}^{\infty}\frac{\left(y-1\right)^{\frac{\alpha}{2}-\frac{1}{4}}}{y^2}dy$.
\end{Lemma}
\noindent For the proof of this Lemma we refer to \cite{AkilMCRFIssa}.
\begin{Lemma}\label{lemma1-strong}
Assume that $\eta\geq 0$. The solution $(z^n,\,v^n,\,\textsf{V}^n,\,p^n,\,\textsf{P}^n,\,\varphi^n)^{\top}\in D(\mathcal{A})$ of the system \eqref{pro5} satisfies the following
\begin{equation}\label{Est1}
\begin{array}{ll}
\displaystyle
\int_{-\ell_{1}}^{0} |z^n|^2dx\xrightarrow[n\to\infty] {} 0, \quad\mathfrak{C}\int_0^{\ell_2}\int_{\mathbb{R}}(|\xi|^{2} + \eta)|\varphi^n(x,\,\xi)|^2 d\xi dx\xrightarrow[n\to\infty] {} 0,\quad \int_0^{\ell_2}|\textsf{V}^n|^2dx\xrightarrow[n\to\infty] {} 0\\[0.2in]\displaystyle
\int_0^{\ell_2}|\lambda_n v^n|^2dx\xrightarrow[n\to\infty] {} 0\quad\text{and}\quad 
\int_0^{\ell_2}\int_{\mathbb{R}}|\varphi^n(x,\,\xi)|^2 d\xi dx\xrightarrow[n\to\infty] {} 0.
\end{array}
\end{equation}
\end{Lemma}

\begin{proof}
Taking the inner product of \eqref{pro4} with $U$ in $\mathcal{H}$ and using the fact that $\|F_n\|_{\mathcal{H}}\xrightarrow[n\to\infty] {} 0$ and $\|U^n\|_{\mathcal{H}}=1$ , we obtain
\begin{equation*}
\kappa \int_{-\ell_{1}}^{0} |z^n_{x}|^2dx + \mathfrak{C}\int_0^{\ell_2}\int_{\mathbb{R}}(|\xi|^{2} + \eta)|\varphi^n(x,\,\xi)|^2 d\xi dx=-\Re(\left<\mathcal{A}U^n,U^n\right>_{\mathcal{H}})\leq  \|F_n\|_{\mathcal{H}}\|U^n\|_{\mathcal{H}}\xrightarrow[n\to\infty] {} 0.
\end{equation*}
Then,  the above results leads to
\begin{equation*}
\int_{-\ell_{1}}^{0} |z^n_{x}|^2dx\xrightarrow[n\to\infty] {} 0\quad\text{and}\quad \mathfrak{C}\int_0^{\ell_2}\int_{\mathbb{R}}(|\xi|^{2} + \eta)|\varphi^n(x,\,\xi)|^2 d\xi dx\xrightarrow[n\to\infty] {} 0.
\end{equation*}
By using Poincar\'e inequality and the first limit in the above equation, we obtain the first limit in \eqref{Est1}. \\
Now our aim is to prove the third limit in \eqref{Est1}. 
From \eqref{pro5}$_6$, we get 
$$
\mu(\xi)\abs{\textsf{V}^n}\leq \left(\abs{\la}+\xi^2+\eta\right)\abs{\varphi^n(x,\xi)}+\abs{f^6_n(x,\xi)}.
$$
Multiplying the above inequality by $\left(\abs{\la_n}+\xi^2+\eta\right)^{-2}\abs{\xi}$,  and integrating over $\R$, we get 
\begin{equation}\label{E1}
\mathtt{J}_{1}(\la_n,\eta,\alpha)\abs{\textsf{V}^n}\leq \mathtt{J}_{2}(\la,\eta)\left(\int_{\R}\abs{\xi\varphi^n(x,\xi)}^2d\xi\right)^{\frac{1}{2}}+\mathtt{J}_{3}(\la,\eta)\left(\int_{\R}\abs{f_6^n(x,\xi)}^2d\xi\right)^{\frac{1}{2}},
\end{equation}
where
$\mathtt{J}_{1}(\la,\eta,\alpha), \mathtt{J}_{2}(\la,\eta)\,\text{and}\,\mathtt{J}_{3}(\la,\eta)$ are defined in Lemma \ref{lemI4}.
Applying Young's inequality leads to
\begin{equation*}
\begin{array}{ll}
\displaystyle
\int_{0}^{l_2}\abs{\textsf{V}^n}^2dx
&\displaystyle\leq 2\frac{\mathtt{J}_{2}^2}{\mathtt{J}_{1}^2}\int_{0}^{l_2}\int_{\R}\abs{\xi\varphi^n(x,\xi)}^2d\xi dx+2\frac{\mathtt{J}_{3}^2}{\mathtt{J}_{1}^2}\int_{0}^{l_2}\int_{\R}\abs{f_6^n(x,\xi)}^2d\xi dx\\
\displaystyle
&\displaystyle\leq  \frac{1}{c_1(\abs{\la}+\eta)^{\alpha-1}}\int_{0}^{l_2}\int_{\R}\abs{\xi\varphi^n(x,\xi)}^2d\xi dx+\frac{\sqrt{\pi}}{4}\frac{1}{c_1\left(\abs{\la}+\eta\right)^{\alpha}}\int_{\R}\abs{f_6^n(x,\xi)}^2d\xi dx.
\end{array}
\end{equation*}
Passing to the limit in the above inequality and by using the second limit in \eqref{Est1},   and the fact that $\|F^n\|\xrightarrow[n\to\infty] {} 0$ and that $\lambda_n\xrightarrow[n\to\infty] {}  \omega$, we conclude that the third limit in \eqref{Est1}. \\
Next, from \eqref{pro5}$_2$ we have
$$i\la_n v^n=\textsf{V}^n+f_n^2.$$
It follows that
$$
\int_{0}^{l_2}\abs{\la_n v^n}^2dx\leq 2 \int_{0}^{l_2}\abs{\textsf{V}^n}^2dx+2\int_{0}^{l_2}\abs{f_n^2}^2dx.
 $$
By using the third limit in \eqref{Est1} and that $\|F^n\|_{\mathcal{H}} \xrightarrow[n\to\infty] {} 0$ in the above inequality, the fourth limit is obtained.   Now, for the last limit in \eqref{Est1}, when $\eta >0$ by using the second limit in \eqref{Est1},  it is clear that 
$$\int_0^{\ell_2}\int_{\mathbb{R}}|\varphi^n(x,\,\xi)|^2 d\xi \leq \int_0^{\ell_2}\int_{\mathbb{R}}(\xi^2+\eta)|\varphi^n(x,\,\xi)|^2 d\xi \xrightarrow[n\to\infty] {} 0.$$
However, when $\eta=0$, multiplying \eqref{pro5}$_6$ by $-i\la_n^{-1} \overline{\varphi^n}$ and integrating over $(0,\ell_2)\times \mathbb{R}$, leads to
\begin{equation}\label{phi}
\begin{array}{ll}
\displaystyle
\int_0^{\ell_2}\int_{\mathbb{R}}|\varphi^n(x,\xi)|^2d\xi dx=\Re\left(i\la_n ^{-1}\int_0^{\ell_2}\int_{\mathbb{R}}|\xi\varphi^n(x,\xi)|^2d\xi dx\right)\\
\displaystyle-\Re\left(i\la_n ^{-1}\int_0^{\ell_2}\textsf{V}^n\int_{\mathbb{R}} \mu(\xi)\overline{\varphi^n}(x,\xi)d\xi dx\right)
-\Re\left(i\la_n ^{-1}\int_0^{\ell_2}\int_{\mathbb{R}} f^6_n(x,\xi)\overline{\varphi^n}(x,\xi)d\xi dx\right).
\end{array}
\end{equation}
For the estimation of the last two terms in the above equation,  using Cauchy-Schwarz  and Young's inequalities, we get
\begin{equation}\label{phi1}
\left|\Re\left(i\la_n ^{-1}\int_0^{\ell_2}\textsf{V}^n\int_{\mathbb{R}} \mu(\xi)\varphi(x,\xi)d\xi  dx\right)\right|\leq \frac{\mathfrak{J}_0(\alpha)}{|\la_n|} \left(\int_0^{\ell_2}\int_{\mathbb{R}}(\xi^2+1)|\varphi^n(x,\xi)|^2d\xi dx\right)^{\frac{1}{2}}\left(\int_0^{\ell_2}|\textsf{V}^n|^2dx\right)^{\frac{1}{2}}
\end{equation}
\begin{equation}\label{phi2}
\left|\Re\left(i\la_n ^{-1}\int_0^{\ell_2}\int_{\mathbb{R}} f^6_n(x,\xi)\varphi^n(x,\xi)d\xi dx\right)\right|\leq \frac{1}{2\la_n^2}\int_0^{\ell_2}\int_{\mathbb{R}} |f_n^6(x,\xi)|^2d\xi dx+\frac{1}{2}\int_0^{\ell_2}\int_{\mathbb{R}} |\varphi^n(x,\xi)|^2d\xi dx
\end{equation}
where $\displaystyle \mathfrak{J}_0(\alpha)=\left(\int_{0}^{\infty}\frac{|\xi|^{2\alpha-1}}{\xi^2+1}d\xi \right)$.  We have  $\displaystyle \frac{|\xi|^{2\alpha-1}}{\xi^2+1}\isEquivTo {0} \frac{1}{\xi^{1-2\alpha}}$ and  $\displaystyle \frac{\abs{\xi}^{2\alpha-1}}{\abs{\xi}^2+1}\isEquivTo{+\infty}\frac{1}{\abs{\xi}^{3-2\alpha}}$ and thus $ \mathfrak{J}_0(\alpha)$  is well defined since $0<\alpha<1$.
Combining \eqref{phi1} and \eqref{phi2} with \eqref{phi}, and using the second and third limits in \eqref{Est1} and the fact that $\|U^n\|_{\mathcal{H}}=1$,  $\|F_n\|_{\mathcal{H}}\xrightarrow[n\to\infty] {} 0$ and $\la_n \xrightarrow[n\to\infty] {}  \omega$,  we can derive the last limit in \eqref{Est1}. Thus, the proof is completed.
\end{proof}

\begin{Lemma}\label{lemma2-strong}
Assume that $\eta\geq 0$. The solution $(z^n,\,v^n,\,\textsf{V}^n,\,p^n,\,\textsf{P}^n,\,\varphi^n)^{\top}\in D(\mathcal{A})$ of the system \eqref{pro5} satisfies the following
\begin{equation}\label{Est2}
\begin{array}{ll}
\displaystyle
\int_{0}^{\ell_2} |v_x^n|^2dx\xrightarrow[n\to\infty] {} 0,\quad \int_0^{\ell_2}| \textsf{P}^n|^2dx\xrightarrow[n\to\infty] {} 0 \\[0.1in]\displaystyle
\quad\text{and}\quad
 \int_0^{\ell_2}|p_x^n|^2dx\xrightarrow[n\to\infty] {} 0.
\end{array}
\end{equation}
\end{Lemma}
\begin{proof}
Multiply \eqref{e1} by $\overline{v}$,  integrate over $(0,\ell_2)$ and take the real part, we obtain
\begin{equation}\label{e3}
\begin{array}{ll}
\displaystyle
\rho \int_0^{\ell_2}|\la_nv^n|^2dx-\chi_1\int_0^{\ell_2}|v_x^n|dx-\Re\left(\chi_1 v_x^n(0)\overline{v^n}(0)\right)+\Re\left(\gamma\mu \int_0^{\ell_2}\la_n p^n \la_n \overline{v}^ndx\right)\\
\displaystyle
-\Re\left(\mathfrak{C}\int_0^{\ell_2}\overline{v^n}\int_{\mathbb{R}}\mu(\xi)\varphi^n(x,\xi)d\xi \right)=\Re\left( \int_0^{\ell_2}F^1_n\overline{v^n}dx\right)+\Re\left(\gamma\int_0^{\ell_2} F^2_n\overline{v^n}dx\right).
\end{array}
\end{equation}
We have that $\|U^n\|_{\mathcal{H}}=1$ and $\|F_n\|_{\mathcal{H}}\xrightarrow[n\to\infty] {} 0$ which implies that $(\la p)$ is uniformly bounded in $L^2(0,\ell_2)$. Thus, using Cauchy-Schwarz inequality, Lemma \ref{lemma1-strong}, and the fact that $\la_n\xrightarrow[n\to\infty] {} \omega$, we obtain the following
\begin{equation}\label{e4}
\left|\Re\left(\gamma\mu \int_0^{\ell_2}\la_n p^n \la_n \overline{v^n}dx\right)\right|\leq \gamma\mu\left(\int_0^{\ell_2}|\la_n p^n|^2dx\right)^{\frac{1}{2}}\left(\int_0^{\ell_2}|\la_n v^n|^2dx\right)^{\frac{1}{2}}\xrightarrow[n\to\infty] {} 0,
\end{equation}

\begin{equation}\label{e5*}
\Re\left( \int_0^{\ell_2}F^1_n\overline{v^n}dx\right)\leq \rho \left(\int_0^{\ell_2}2|f_n^3|^2dx+2|\la_n|^2 \int_0^{\ell_2}|f_n^2|^2dx\right)^{\frac{1}{2}} \left(\int_0^{\ell_2}|v^n|^2dx\right)^{\frac{1}{2}}
\xrightarrow[n\to\infty] {} 0,
\end{equation}
Similarly,
\begin{equation}\label{e5}
\Re\left(\gamma\int_0^{\ell_2} F^2_n\overline{v^n}dx\right)\xrightarrow[n\to\infty] {} 0,
\end{equation}
Now, for the last term on the left hand side of \eqref{e3}, we have
\begin{equation}
\begin{array}{ll}
\displaystyle
\left|\Re\left(\mathfrak{C}\int_0^{\ell_2}\overline{v^n}\int_{\mathbb{R}}\mu(\xi)\varphi^n(x,\xi)d\xi \right)\right|\leq
2 \mathfrak{C}\int_0^{\ell_2} |v^n|\int_{0}^{\infty}\mu(\xi)|\varphi^n(x,\xi)|d\xi dx\leq 2\mathfrak{C}\left(\mathfrak{M}_1+\mathfrak{M}_2\right),
\end{array}
\end{equation}
where 
$$\mathfrak{M}_1=\int_0^{\ell_2}\overline{v^n}\left(\int_{0}^1\mu(\xi)\varphi^n(x,\xi)d\xi\right)dx \quad\text{and}\quad \mathfrak{M}_2=\int_0^{\ell_2}\overline{v^n}\left(\int_{1}^{\infty}\mu(\xi)\varphi^n(x,\xi)d\xi\right)dx, $$

\begin{equation*}
\begin{array}{ll}
\displaystyle
\left|\mathfrak{M}_1\right| \leq  \int_0^{\ell_2} |v^n|\left(\int_{0}^{1}\frac{\mu(\xi)}{\sqrt{\xi^2+1}}\sqrt{\xi^2+1}| \varphi^n(x,\xi)|d\xi\right) dx\\
\displaystyle
 \leq  \mathfrak{J}_1(\alpha,\eta) \left(\int_0^{\ell_2}|v^n|^2dx\right)^{\frac{1}{2}} \left( \int_0^{\ell_2}\int_{0}^{1}(\xi^2+1)\abs{\varphi^n(x,\xi)}^2d\xi dx\right)^{\frac{1}{2}},
\end{array}
\end{equation*}

\begin{equation*}
\begin{array}{ll}\displaystyle
\left|\mathfrak{M}_2\right| \leq \int_0^{\ell_2} |v^n|\left(\int_{1}^{\infty}\frac{\abs{\mu(\xi)}\sqrt{\xi^2+\eta}}{\sqrt{\xi^2+\eta}}|\varphi^n(x,\xi)|d\xi\right)^2 dx \\
\displaystyle
\leq \mathfrak{J}_2(\alpha,\eta) \left(\int_0^{\ell_2}|v^n|^2dx\right)^{\frac{1}{2}} \left( \int_0^{\ell_2}\int_{1}^{\infty}(\xi^2+\eta)\abs{\varphi^n(x,\xi)}^2d\xi dx\right)^{\frac{1}{2}},
\end{array}
\end{equation*}
where $\displaystyle\mathfrak{J}_1(\alpha)=\left(\int_0^1\frac{|\xi|^{2\alpha-1}}{\xi^2+1}d\xi\right)^{\frac{1}{2}}$ and $\displaystyle\mathfrak{J}_2(\alpha,\eta)=\left(\int_{1}^{\infty}\frac{\abs{\xi}^{2\alpha-1}}{\abs{\xi}^2+\eta}d\xi\right)^{\frac{1}{2}}$.  We have  $\displaystyle \frac{\xi^{2\alpha-1}}{\xi^2+1}\isEquivTo {0} \frac{1}{\xi^{1-2\alpha}}$ and  $\displaystyle \frac{\abs{\xi}^{2\alpha-1}}{\abs{\xi}^2+\eta}\isEquivTo{+\infty}\frac{1}{\abs{\xi}^{3-2\alpha}}$.
Since $0<\alpha<1$ and $\eta\geq 0$, then $\mathfrak{J}_1(\alpha)$  and $\mathfrak{J}_2(\alpha,\eta)$ are well defined. Thus,  using Lemma \ref{lemma1-strong} and that $\|U\|_{\mathcal{H}}=1$, we deduce that 
$$|\mathfrak{M}_1|\xrightarrow[n\to\infty] {} 0\quad\text{and}\quad|\mathfrak{M}_2|\xrightarrow[n\to\infty] {} 0.
$$
Therefore,
 \begin{equation}\label{e6}
\begin{array}{ll}
\displaystyle
\left|\Re\left(\mathfrak{C}\int_0^{\ell_2}\overline{v^n}\int_{\mathbb{R}}\mu(\xi)\varphi^n(x,\xi)d\xi \right)\right|\xrightarrow[n\to\infty] {} 0.
\end{array}
\end{equation}
Gagliardo-Nirenberg inequality implies that
\begin{equation}
|z^n_x(0)|\lesssim \|z^n_{xx}\|^{\frac{1}{2}}_{L^2(0,\ell_2)} \|z^n_{x}\|^{\frac{1}{2}}_{L^2(0,\ell_2)}+ \|z^n_{x}\|_{L^2(0,\ell_2)}.
\end{equation}
By using the fact that $ \|z^n_{xx}\|^{\frac{1}{2}}_{L^2(0,\ell_2)}\leq |\omega |^{\frac{1}{2}} \|z^n\|^{\frac{1}{2}}_{L^2(0,\ell_2)}+\|f^1_n\|^{\frac{1}{2}}_{L^2(0,\ell_2)}$ and Lemma \ref{lemma1-strong} in the above inequality we obtain that 
\begin{equation}\label{z_x}
|z^n_x(0)|\xrightarrow[n\to\infty] {} 0.
\end{equation}
Now, from the transmission conditions we have 
\begin{eqnarray*}
\chi_1 v^n_x(0)+\gamma(\gamma\beta v^n_x(0)-\beta p^n_x(0))=\kappa z^n_x(0),\\
\beta p^n_x(0)=\gamma\beta v^n_x(0).
\end{eqnarray*}
Combining the above conditions with the result in \eqref{z_x} we deduce that 
\begin{equation}\label{v_x}
|v^n_x(0)|\xrightarrow[n\to\infty] {}  0  \quad\text{and}\quad |p^n_x(0)|\xrightarrow[n\to\infty] {}0.
\end{equation}
Thus, using the fact that $\displaystyle|v^n(0)|^2\leq \int_0^{\ell_2}|v^n_x|^2dx\leq \|U^n\|_{\mathcal{H}}=1$, and from \eqref{v_x} we obtain that
\begin{equation}\label{e7}
\left|\Re\left(\chi_1 v_x^n(0)\overline{v^n}(0)\right)\right|\xrightarrow[n\to\infty] {}  0.
\end{equation}
Thus, using \eqref{e4}-\eqref{e5}, \eqref{e6}, and  \eqref{e7}, we obtain the first limit in \eqref{Est2}.\\
Next, multiply \eqref{e1} by $\overline{p}$,  integrate over $(0,\ell_2)$ and take the real part, we obtain
\begin{equation}
\begin{array}{ll}
\displaystyle
\Re\left(\rho\int_0^{\ell_2}\la v^n\la \overline{p^n}dx\right)-\Re\left(\chi_1\int_0^{\ell_2}v^n_x\overline{p^n_x}dx\right)-\Re\left(\chi_1 v_x(0)\overline{p^n}(0)\right)-\gamma\mu \int_0^{\ell_2}|\la_n p^n|^2dx\\
\displaystyle
-\Re\left(\mathfrak{C}\int_0^{\ell_2}\overline{p^n}\int_{\mathbb{R}}\mu(\xi)\varphi^n(x,\xi)d\xi \right)=\Re\left( \int_0^{\ell_2}F^1_n\overline{v^n}dx\right)+\Re\left(\gamma\int_0^{\ell_2} F^2_n\overline{p^n}dx\right).
\end{array}
\end{equation}
The terms of the above equation can be estimated by using similar calculations as before and using the fact that $\la p$ and $p_x$ are uniformly bounded in $L^2(0,\ell_2)$, Lemma \ref{lemma1-strong}, the first limit in \eqref{lemma2-strong},  and \eqref{v_x}, we obtain that
\begin{equation}\label{e8}
\int_0^{\ell_2}|\la_n p^n|^2dx\xrightarrow[n\to\infty] {}  0.
\end{equation}
From equation \eqref{pro4}$_4$, and by using that $\|F_n\|\xrightarrow[n\to\infty] {}  0$, we get 
\begin{equation}
\int_0^{\ell_2}| \textsf{P}^n|^2dx\leq \int_0^{\ell_2}|\la_n p^n|^2dx+\int_0^{\ell_2}| f^4_n|^2dx\xrightarrow[n\to\infty] {}  0.
\end{equation}
Thus, the second limit in \eqref{Est2} is proved.  Finally,  multiply \eqref{e2} by $\overline{p}$ and integrate over $(0,\ell_2)$, we get
\begin{equation}\label{e9}
\begin{array}{ll}
\displaystyle
\mu\chi \int_0^{\ell_2}|\la_n p^n|^2dx-\chi_1\beta\int_0^{\ell_2}|p_x^n|^2dx-\Re\left(\chi_1\beta p^n_x(0)p^n(0)\right)+\Re\left(\rho\beta\gamma\int_0^{\ell_2} \la v^n\la \overline{p^n}dx\right)\\
\displaystyle
-\Re\left(\gamma\beta\mathfrak{C}\int_0^{\ell_2}\overline{p^n}\int_{\mathbb{R}}\mu(\xi)\varphi(x,\xi)d\xi \right)=\Re\left( \gamma\beta
\int_0^{\ell_2}F^1_n\overline{p^n}dx\right)+\Re\left(\gamma\int_0^{\ell_2} F^2_n\overline{p^n}dx\right).
\end{array}
\end{equation}
From \eqref{e8} and the first limit in \eqref{Est2}, we get
\begin{equation}\label{e10}
\mu\chi \int_0^{\ell_2}|\la_n p^n|^2dx\xrightarrow[n\to\infty] {}  0\quad\text{and}\quad \Re\left(\rho\beta\gamma\int_0^{\ell_2} \la v^n\la \overline{p^n}dx\right)\xrightarrow[n\to\infty] {}  0.
\end{equation}
By using \eqref{v_x} and the fact that $|p(0)|\leq \|p_x\|_{L^2(0,\ell_2)}$, we deduce that 
\begin{equation}\label{e11}
|\Re\left(\chi_1\beta p^n_x(0)p^n(0)\right)|\xrightarrow[n\to\infty] {}  0.
\end{equation}
Cauchy-Schwarz, \eqref{Est1} and the fact that $\|F_n\|_{\mathcal{H}}$, leads to 
\begin{equation}
\Re\left(\gamma\beta\mathfrak{C}\int_0^{\ell_2}\overline{p^n}\int_{\mathbb{R}}\mu(\xi)\varphi^n(x,\xi)d\xi \right)\xrightarrow[n\to\infty] {}  0,
\end{equation}
and
\begin{equation}\label{e12}
	\Re\left( \gamma\beta \int_0^{\ell_2}F^1_n\overline{p^n}dx\right)+\Re\left(\gamma\int_0^{\ell_2} F^2_n\overline{p^n}dx\right)\xrightarrow[n\to\infty] {}  0.
\end{equation}
Therefore, using \eqref{e10}-\eqref{e12} in \eqref{e10} we obtain the last limit in \eqref{Est2}.
\end{proof}

\textbf{Proof of Proposition \ref{proposition2}} From Lemmas \ref{lemma1-strong} and \ref{lemma2-strong}, we obtain that $\|U^n\|_{\mathcal{H}}\to 0$ as $n\to 0$ which contradicts that $\|U^n\|_{\mathcal{H}}=1$ in \eqref{pro3}. Then, $i\R^{\ast}\subseteq \rho(\mathcal{A})$  holds true and the proof is complete.\\

\textbf{Proof of Theorem \ref{Strong1}}
From Proposition \ref{proposition2}, we have $i\R\subseteq \rho(\mathcal{A})$ if $\eta>0$ and consequently $\sigma(\mathcal{A})\cap i\R=\emptyset$ and $i\R^{\ast}\subseteq \rho(\mathcal{A})$ if $\eta=0$ and consequently $\sigma(\mathcal{A})\cap i\R=\{0\}$  .  Therefore, according to Arendt-Batty Theorem, we get that the C$_0$-semigroup of contraction $(e^{t\mathcal{A}})_{t\geq 0}$ is strongly stable. The proof is thus complete.

\section{Polynomial Stability}\label{Polynomaial stability}
\noindent The aim of this section is to study the polynomial stability of system \eqref{209} in the case $\eta>0$.  Our main result in this part is the following theorem. 
\begin{Theorem}\label{polynomial theorem}
Assume that $\eta>0$. The $C_0-$ semigroup $(e^{t\mathcal{A}})_{t\geq 0}$ is polynomially stable; i.e. there exists constant $C>0$ such that for every $U_0\in D(\mathcal{A})$, we have 
\begin{equation}\label{Energypol}
\mathcal{E}(t)\leq \frac{C}{t^{\frac{2}{1-\alpha}}}\|U_0\|^2_{D(\mathcal{A})},\quad t>0,\,\forall\, U_0\in D(\mathcal{A}).
\end{equation}
\end{Theorem}
\noindent According to Theorem of Borichev and Tomilov \cite{Borichev01} (see also \cite{RaoLiu01} and \cite{Batty01}), in order to prove Theorem \ref{polynomial theorem} we need to prove that the following two conditions hold:
\begin{equation}\label{R1}\tag{${\rm{P_1}}$}
i\R\subset \rho(\mathcal{A}),
\end{equation}
and
\begin{equation}\label{R2}\tag{${\rm{P_2}}$}
\limsup_{|\la|\to\infty}\frac{1}{|\la|^{1-\alpha}}\left\|(i\la I-\mathcal{A})^{-1}\right\|_{\mathcal{L}(\mathcal{H})}<\infty
\end{equation}
are satisfied. \\
Note that,  when $\eta > 0,$ from Proposition \ref{proposition2} we have $i\mathbb{R}\subseteq\rho({\mathcal{A}}).$ Therefore, according to the Borichev-Tomilov Theorem, polynomial stability is equivalent to having the conditions \eqref{R1} and \eqref{R2} satisfied.\\
\noindent Since condition \eqref{R1} is already proved (see Proposition  \ref{proposition2}), we still need to prove condition \eqref{R2}. For this purpose we will use an argument of contradiction. Suppose that \eqref{R2} is false, then there exists $$\left\{\left(\la_n,U_n:=(z^n,\,v^n,\,\textsf{V}^n,\,p^n,\,\textsf{P}^n,\,\varphi^n)^\top\right)\right\}\subset \R^{\ast}\times D(\mathcal{A})$$ with 
\begin{equation}\label{pol1}
\abs{\la_n}\to +\infty \quad \text{and}\quad \|U_n\|_{\mathcal{H}}=\|(z^n,\,v^n,\,\textsf{V}^n,\,p^n,\,\textsf{P}^n,\,\varphi^n)\|_{\mathcal{H}}=1, 
\end{equation}
such that 
\begin{equation}\label{pol2}
\left(\la_n^{1-\alpha}\right)(i\la_n I-\mathcal{A})U^n=F_n:=(f^1_n, f^2_n, f^3_n, f^4_n, f^5_n, f^6_n(\cdot,\xi))^{\top}\to 0\quad\text{in}\quad\mathcal{H}.
\end{equation}
For simplicity, we drop the index $n$. Equivalently, from \eqref{pol2}, we have 
\begin{equation}\label{pol3}
\left\{
\begin{array}{ll}
\displaystyle
i\lambda z - \kappa z_{xx} =\frac{f^1}{\la^{1-\alpha}},  \\[0.1in]
\displaystyle
i\lambda v - \textsf{V} = \frac{f^2}{\la^{1-\alpha}} ,  \\[0.1in]
\displaystyle
i\lambda\rho\textsf{V} - \chi v_{xx} + \gamma \beta p_{xx} + \mathfrak{C} \displaystyle\int_{\mathbb{R}}\mu(\xi)\varphi(x,\xi)d\xi = \rho \frac{f^3}{\la^{1-\alpha}} ,  \\[0.1in]
\displaystyle
i\lambda p - \textsf{P} = \frac{f^4}{\la^{1-\alpha}},  \\[0.1in]
\displaystyle
i\lambda\mu\textsf{P} - \beta p_{xx} + \gamma \beta v_{xx} =\mu \frac{f^5}{\la^{1-\alpha}} , \\[0.1in]
\displaystyle
(|\xi|^{2} + \eta + i\lambda)\varphi(\xi) - \mu(\xi)\textsf{V} = \frac{f^6(\cdot,\xi)}{\la^{1-\alpha}} ,\quad \forall \,\xi\in \mathbb{R}.
\end{array}\right.
\end{equation}
Here we will check the condition \eqref{R2} by finding a contradiction with \eqref{pol1} by showing $\|U\|_{\mathcal{H}}=o(1)$. For clarity, we divide the proof into several Lemmas.
\begin{Lemma}\label{lemma1 pol}
Assume that $\eta>0$. Then, the  solution $(z,\,v,\,\textsf{V},\,p,\,\textsf{P},\,\varphi)^{\top}\in D(\mathcal{A})$ of the system \eqref{pol3} satisfies the following asymptotic behavior estimations 
\begin{equation}\label{Pol Est1}
\begin{array}{ll}
\displaystyle
\int_{-\ell_{1}}^{0} |z|^2dx=\frac{o(1)}{\la^{1-\alpha}}, \quad\mathfrak{C}\int_0^{\ell_2}\int_{\mathbb{R}}(|\xi|^{2} + \eta)|\varphi(x,\,\xi)|^2 d\xi dx=\frac{o(1)}{\la^{1-\alpha}},\\[0.1in]\displaystyle
 \int_0^{\ell_2}|\textsf{V}|^2dx=o(1),
\quad\text{and}\quad
\int_0^{\ell_2}|\lambda v|^2dx=o(1).
\end{array}
\end{equation}
\end{Lemma}
\begin{proof}
For the clarity of the proof, we divide it into several steps.\\ 
\textbf{Step 1.} Taking the inner product of \eqref{pol2} with $U$ in $\mathcal{H}$, then using \eqref{pol1} and the fact that $U$ is uniformly bounded in $\mathcal{H}$, we get 
\begin{equation*}
\kappa \int_{-\ell_{1}}^{0} |z_{x}|^2dx + \mathfrak{C}\int_0^{\ell_2}\int_{\mathbb{R}}(|\xi|^{2} + \eta)|\varphi(x,\,\xi)|^2 d\xi dx=-\Re(\left<\mathcal{A}U,U\right>_{\mathcal{H}})\leq  \|F\|_{\mathcal{H}}\|U\|_{\mathcal{H}}=\frac{o(1)}{\la^{1-\alpha}}.
\end{equation*}
Then,  the above results lead to
\begin{equation*}
\int_{-\ell_{1}}^{0} |z_{x}|^2dx=\frac{o(1)}{\la^{1-\alpha}}\quad\text{and}\quad \mathfrak{C}\int_0^{\ell_2}\int_{\mathbb{R}}(|\xi|^{2} + \eta)|\varphi(x,\,\xi)|^2 d\xi dx=\frac{o(1)}{\la^{1-\alpha}}.
\end{equation*}
By using Poincar\'e inequality and the first estimation in the above equation, we obtain the first estimation in \eqref{Pol Est1}. \\
\noindent\textbf{Step 2.} Our aim here is to prove the third estimation in \eqref{Pol Est1}.\\
From \eqref{pol3}$_6$, we get 
$$
\mu(\xi)\abs{\textsf{V}}\leq \left(\abs{\la}+\xi^2+\eta\right)\abs{\varphi(x,\xi)}+\frac{1}{|\la|^{1-\alpha}}\abs{f^6(x,\xi)}.
$$
Multiplying the above inequality by $\left(\abs{\la_n}+\xi^2+\eta\right)^{-2}\abs{\xi}$,  integrating over $\R$,  and proceeding as in Lemma \ref{lemma1-strong},  and by using the first and second estimations in \eqref{Pol Est1},  we get 
\begin{equation*}
\begin{array}{lll}
\displaystyle
\int_{0}^{l_2}\abs{\textsf{V}}^2dx
&\displaystyle\leq  \frac{1}{c_1(\abs{\la}+\eta)^{\alpha-1}}\int_{0}^{l_2}\int_{\R}\abs{\xi\varphi(x,\xi)}^2d\xi dx+\frac{\sqrt{\pi}}{4|\la|^{1-\alpha}}\frac{1}{c_1\left(\abs{\la}+\eta\right)^{\alpha}}\int_{\R}\abs{f^6(x,\xi)}^2d\xi dx\\
&\displaystyle \leq
\displaystyle
 \frac{1}{c_1(\abs{\la}+\eta)^{\alpha-1}}\frac{o(1)}{|\la|^{1-\alpha}}+\frac{\sqrt{\pi}}{4|\la|^{1-\alpha}}\frac{1}{c_1\left(\abs{\la}+\eta\right)^{\alpha}}\frac{o(1)}{|\la|^{1-\alpha}}.
\end{array}
\end{equation*}
Hence,  from the above equation we get the third estimation in \eqref{Pol Est1}. \\
\textbf{Step 3.}  From equation \eqref{pol3}$_2, $ we have that
$$
\int_{0}^{l_2}\abs{\la v}^2dx\leq 2\int_{0}^{l_2}\abs{\textsf{V}}^2dx+2\abs{\la}^{-2+2\alpha}\int_{0}^{l_2}\abs{f^2}^2dx.
 $$
Using Step 2, the fact that $\|f^2\|_{H^1_R(0,\ell_2)}=o(1)$, and that $\alpha\in (0,1)$, we get the last estimation in \eqref{Pol Est1}.
\end{proof}
\vspace{0.3cm}

\noindent Inserting \eqref{pol3}$_2$ and \eqref{pol3}$_4$ into \eqref{pol3}$_3$ and \eqref{pol3}$_5$, and proceeding in a similar way as in \eqref{E2}-\eqref{e2}, we obtain
\begin{eqnarray}
\displaystyle
\la _n^2\rho v+\chi _1v_{xx}+\la _n^2\gamma \mu p-\mathfrak{C}\int_{\mathbb{R}}\mu(\xi)\varphi(x,\xi)d\xi =\mathcal{G}^1+\gamma \mathcal{G}^2 ,\label{p1}\\
\displaystyle
\la _n^2 \mu\chi p+\chi_1\beta p_{xx}+\la _n^2\rho\gamma\beta v-\gamma\beta\mathfrak{C}\int_{\mathbb{R}}\mu(\xi)\varphi(x,\xi)d\xi =\gamma\beta \mathcal{G}^1+\chi \mathcal{G}^2,\label{p2}
\end{eqnarray}
where $\mathcal{G}^1=-\la^{-(1-\alpha)}(\rho f^3+i\la\rho f^2)$ and $\mathcal{G}^2=-\la^{-(1-\alpha)}(\mu f^5+i\la\mu f^4)$ .

\begin{Lemma}\label{lemma2 pol}
Assume that $\eta>0$. Then, the  solution $(z,\,v,\,\textsf{V},\,p,\,\textsf{P},\,\varphi)^{\top}\in D(\mathcal{A})$ of the system \eqref{pol3} satisfies the following asymptotic behavior estimations 
\begin{equation}\label{Pol Est2}
\begin{array}{ll}
\displaystyle
\int_{0}^{\ell_2} |v_x|^2dx=o(1),\quad \int_0^{\ell_2}| \textsf{P}|^2dx=o(1)
\quad\text{and}\quad
 \int_0^{\ell_2}|p_x|^2dx=o(1).
\end{array}
\end{equation}
\end{Lemma}

\begin{proof}
Multiply \eqref{p1} by $\overline{v}$,  integrate over $(0,\ell_2)$ and take the real part, we obtain
\begin{equation}\label{p3}
\begin{array}{ll}
\displaystyle
\rho \int_0^{\ell_2}|\la v|^2dx-\chi_1\int_0^{\ell_2}|v_x|^2dx-\Re\left(\chi_1 v_x(0)\overline{v}(0)\right)+\Re\left(\gamma\mu \int_0^{\ell_2}\la_n p \la \overline{v}dx\right)\\
\displaystyle
-\Re\left(\mathfrak{C}\int_0^{\ell_2}\overline{v}\int_{\mathbb{R}}\mu(\xi)\varphi(x,\xi)d\xi \right)=\Re\left( \int_0^{\ell_2} \mathcal{G}^1\overline{v}dx\right)+\Re\left(\gamma\int_0^{\ell_2} \mathcal{G}^2\overline{v}dx\right).
\end{array}
\end{equation}
We have that $\|U^n\|_{\mathcal{H}}=1$ and $\|F_n\|_{\mathcal{H}} =o(1)$ which implies that $\la p$ is uniformly bounded in $L^2(0,\ell_2)$. Thus, using Cauchy-Schwarz inequality,  and Lemma \ref{lemma1 pol}, we obtain the following
\begin{equation}\label{p4}
\left|\Re\left(\gamma\mu \int_0^{\ell_2}\la p \la  \overline{v} dx\right)\right|\leq \gamma\mu\left(\int_0^{\ell_2}|\la p|^2dx\right)^{\frac{1}{2}}\left(\int_0^{\ell_2}|\la v|^2dx\right)^{\frac{1}{2}}=o(1),
\end{equation}
\begin{equation}\label{p5}
\left|\Re\left( \int_0^{\ell_2}\mathcal{G}^1\overline{v}dx\right)\right|\leq \frac{\rho}{\la^{1-\alpha}}\left(\int_0^{\ell_2}\left(2|f^3|^2+2|f^2|^2\right)dx\right)^{\frac{1}{2}}\left(\int_0^{\ell_2}|v|^2dx\right)^{\frac{1}{2}}=\frac{o(1)}{\la ^{1-\alpha}},
\end{equation}
similarly we can obtain
\begin{equation}\label{p6}
\left|\Re\left(\gamma\int_0^{\ell_2} \mathcal{G}^2\overline{v}dx\right)\right| =\frac{o(1)}{\la ^{1-\alpha}}.
\end{equation}
Now, for the last term on the left hand side of \eqref{p3}, we have
\begin{equation}\label{p7}
\begin{array}{ll}
\displaystyle
\left|\Re\left(\mathfrak{C}\int_0^{\ell_2}\overline{v}\int_{\mathbb{R}}\mu(\xi)\varphi(x,\xi)d\xi \right)\right|\leq \mathfrak{C}\int_0^{\ell_2} \overline{v}\left(\int_{\R}\frac{\abs{\mu(\xi)}\sqrt{\xi^2+\eta}}{\sqrt{\xi^2+\eta}}\varphi(x,\xi)d\xi\right)^2 dx \\
\displaystyle
\leq \mathfrak{C} \mathfrak{J}(\alpha,\eta) \left(\int_0^{\ell_2}|v|^2dx\right)^{\frac{1}{2}} \left( \int_0^{\ell_2}\int_{\R}(\xi^2+\eta)\abs{\varphi(x,\xi)}^2d\xi dx\right)^{\frac{1}{2}}=\frac{o(1)}{\la^{\frac{3}{2}-\frac{\alpha}{2}}},
\end{array}
\end{equation}
where $\displaystyle\mathfrak{J}(\alpha,\eta)=\left(\int_{\R}\frac{\abs{\xi}^{2\alpha-1}}{\abs{\xi}^2+\eta}d\xi\right)^{\frac{1}{2}}$  which is well defined since $0<\alpha<1$.  \\
Using \eqref{pol3}$_1$, we get
$$
\|z_{xx}\|_{L^2(-\ell_1,0)} \leq \|\la z\|_{L^2(-\ell_1,0)}+\la^{-1+\alpha}\|f^1\|_{L^2(-\ell_1,0)}\leq  \frac{o(1)}{\la^{-\frac{1}{2}-\frac{\alpha}{2}}}.
$$
 Gagliardo-Nirenberg inequality with the above equation implies that
\begin{equation}\label{z_x0}
\begin{array}{ll}
|z_x(0)|\leq \|z_{xx}\|^{\frac{1}{2}}_{L^2(-\ell_1,0)} \|z_{x}\|^{\frac{1}{2}}_{L^2(-\ell_1,0)}+ \|z_{x}\|_{L^2(-\ell_1, 0)}
=\frac{o(1)}{\la^{-\frac{\alpha}{2}}}.
\end{array}
\end{equation}
From the transmission conditions we can deduce that 
\begin{equation}\label{v_x(0)}
|v_x(0)|=\frac{o(1)}{\la^{-\frac{\alpha}{2}}}\quad\text{and}\quad|p_x(0)|=\frac{o(1)}{\la^{-\frac{\alpha}{2}}}.
\end{equation}
Using Lemma \ref{lemma1 pol},  we obtain that $|z(0)|\leq \|z_{x}\|_{L^2(0,\ell_2)}=\frac{o(1)}{\la^{\frac{1-\alpha}{2}}}$.  Using this result and the transmission condition we get
\begin{equation*}
|v(0)|\leq \frac{1}{|\la|}|\textsf{V}(0)|+\frac{1}{\la ^{2-\alpha}}|f^2(0)|\leq \frac{1}{|\la|}|z(0)|+\frac{1}{\la ^{2-\alpha}}|f^2(0)| =\frac{o(1)}{\la^{\min(\frac{3}{2}-\frac{\alpha}{2},2-\alpha)}}=\frac{o(1)}{\la^{\frac{3}{2}-\frac{\alpha}{2}}}.
\end{equation*}
Thus,  from the above equation and \eqref{v_x(0)}, we obtain 
\begin{equation}\label{v_xv}
|\Re\left(\chi_1 v_x(0)\overline{v}(0)\right)|=\frac{o(1)}{\la^{\frac{3}{2}-\alpha}}
\end{equation}
Therefore, using \eqref{p4}-\eqref{p7} and \eqref{v_xv} in \eqref{p3}, we obtain 
$$\int_{0}^{\ell_2} |v_x|^2dx=o(1).$$
Next, multiply \eqref{p1} by $\overline{p}$ and integrate over $(0,\ell_2)$, we obtain
\begin{equation}\label{p8}
\begin{array}{ll}
\displaystyle
\Re\left(\rho\int_0^{\ell_2}\la v\la \overline{p}dx\right)-\Re\left(\chi_1\int_0^{\ell_2}v_x\overline{p_x}dx\right)-\Re\left(\chi_1 v_x(0)\overline{p}(0)\right)-\gamma\mu \int_0^{\ell_2}|\la p|^2dx\\
\displaystyle
-\Re\left(\mathfrak{C}\int_0^{\ell_2}\overline{p}\int_{\mathbb{R}}\mu(\xi)\varphi(x,\xi)d\xi \right)=\Re\left( \int_0^{\ell_2}\mathcal{G}^1\overline{p}dx\right)+\Re\left(\gamma\int_0^{\ell_2} \mathcal{G}^2\overline{p}dx\right).
\end{array}
\end{equation}
The terms of the above equation can be estimated by using similar calculations as above and using the fact that $(\la p)$ and $(p_x)$ are uniformly bounded in $L^2(0,\ell_2)$, Lemma \ref{lemma1 pol}, the first limit in \eqref{lemma2 pol},  and \eqref{v_x(0)}, we obtain that
\begin{equation}\label{p9}
\left\{\begin{array}{lll}
\displaystyle
\left|\Re\left(\gamma\mu \int_0^{\ell_2}\la p \la \overline{v}dx\right)\right|=o(1),\quad \left|\Re\left(\chi_1\int_0^{\ell_2}v_x\overline{p_x}dx\right)\right|=o(1),\\[0.2in]
\displaystyle
\left|\Re\left( \int_0^{\ell_2}\mathcal{G}^1\overline{p}dx\right)\right| 
=\frac{o(1)}{\la ^{1-\alpha}}, \quad \left|\Re\left(\gamma\int_0^{\ell_2} \mathcal{G}^2_n\overline{p}dx\right)\right| =\frac{o(1)}{\la ^{1-\alpha}},\\[0.2in]
\displaystyle
\text{and}\quad\left|\Re\left(\mathfrak{C}\int_0^{\ell_2}\overline{p}\int_{\mathbb{R}}\mu(\xi)\varphi(x,\xi)d\xi \right)\right|=\frac{o(1)}{\la^{\frac{3}{2}-\frac{\alpha}{2}}}.
\end{array}\right.
\end{equation}
By using Gagliardo-Nirenberg we have that $|p(0)|\leq \|p_x\|_{L^2(0,\ell_2)}^{\frac{1}{2}}\|p\|_{L^2(0,\ell_2)}^{\frac{1}{2}}+\|p\|_{L^2(0,\ell_2)}=\frac{O(1)}{\lambda^{\frac{1}{2}}}$. From this result and \eqref{v_x(0)}, we get
\begin{equation}\label{p_x0}
|\Re\left(\chi_1 v_x(0)\overline{p}(0)\right)|=\frac{o(1)}{\la^{\frac{1-\alpha}{2}}}.
\end{equation}
Thus, using \eqref{p9} and \eqref{p_x0} in \eqref{p8} leads to
\begin{equation}\label{p10}
\int_0^{\ell_2}|\la p|^2dx=o(1).
\end{equation}
From equation \eqref{pol2}$_4$, and by using that $\|F\|_{\mathcal{H}}=o(1)$, we get 
\begin{equation}
\int_0^{\ell_2}| \textsf{P}|^2dx\leq 2\int_0^{\ell_2}|\la p|^2dx+\frac{2}{\la^{2(1-\alpha)}}\int_0^{\ell_2}| f^4|^2dx=o(1).
\end{equation}
Thus, the second limit in \eqref{Pol Est2} is proved.  \\
Finally,  multiply \eqref{p2} by $\overline{p}$ and integrate over $(0,\ell_2)$, we get
\begin{equation}\label{e9}
\begin{array}{ll}
\displaystyle
\mu\chi \int_0^{\ell_2}|\la p|^2dx-\chi_1\beta\int_0^{\ell_2}|p_x|^2dx+\Re\left(\chi_1\beta p_x(0)p(0)\right)+\Re\left(\rho\beta\gamma\int_0^{\ell_2} \la v\la \overline{p}dx\right)\\
\displaystyle
-\Re\left(\gamma\beta\mathfrak{C}\int_0^{\ell_2}\overline{p}\int_{\mathbb{R}}\mu(\xi)\varphi(x,\xi)d\xi \right)=\Re\left( \gamma\beta
\int_0^{\ell_2}\mathcal{G}^1\overline{p}dx\right)+\Re\left(\gamma\int_0^{\ell_2} \mathcal{G}^2\overline{p}dx\right).
\end{array}
\end{equation}
By using similar calculation as for the second limit in this Lemma, we can conclude that 
the last limit in \eqref{Pol Est2} is proved.
\end{proof}

\noindent \textbf{Proof of Theorem \ref{polynomial theorem}.} 
From Lemmas \ref{lemma1 pol} and \ref{lemma2 pol} we get that $\|U\|_{\mathcal{H}}=o(1)$, which contradicts \eqref{pol1}. Consequently, condition \eqref{R2} holds. This implies,  that the energy decay estimation \eqref{Energypol} holds.  The proof is thus complete.

\begin{rem}\normalfont
		\leavevmode
	\begin{enumerate}
		\item
		In the absence of fractional derivatives, the system described in this paper exhibits polynomial stability, as demonstrated through the analysis presented. Conversely, Akil (2022) \cite{Akil2022} investigated a similar system involving the coupling of piezoelectric beams with heat acting on the whole domain, where he achieved exponential stability. This contrast implies that when fractional derivatives are not present, for the heat acting on the whole doamin as explored in Akil et al.'s study, offers enhanced stability compared to considering the heat through transmission.
		\item Another comparison can be drawn with the findings in An et al. (2022) \cite{an2022stability}, where the inclusion of boundary fractional derivatives led to a lack of exponential stability and instead resulted in polynomial stability of order $t^{1-\alpha}$, particularly with two boundary dampings. In our system, we similarly introduce fractional damping affecting longitudinal displacement and obtain the same polynomial stability result, even in the absence of heat conduction. However, upon incorporating thermal effects described by Fourier’s law, An et al. achieved exponential stability, contrasting with our findings here. This suggests that in systems involving fractional derivatives, considering two damping on the boundary and and the heat to be acting on the whole domain achieve a better decay rate.
\item In \cite{AkilOzer2024} the authors considered a transmission problem of piezoelectric-elastic system (noted as P/E) with a local damping acting on the elastic part only.  They proved the strong stability under the condition
		$$\frac{\sigma_+}{\sigma_-}\neq \frac{2n_+-1}{2n_{-}-1}, \quad \forall \, n_+, n_-\in \mathbb{N}$$
		where the two positive real numbers $\sigma_+$ and $\sigma_-$ are defined by
		$$\sigma_{\pm}:=\sqrt{\frac{(\rho\beta+\mu\alpha)\pm \sqrt{(\rho\beta-\mu\alpha)^2+4\gamma^2\beta^2\mu \rho}}{2\beta \alpha_1}},$$
		(where $\alpha, \alpha_1$ are $\chi, \chi_1$ in our work). They showed that the stability is polynomial or exponential,  dependening entirely upon the arithmetic nature of a quotient involving all physical parameters and represented by the following two conditions (see \cite{AkilOzer2024}).\\
		$\mathbf{(H_{EXP})}$ Assume that $\frac{\sigma_+}{\sigma_-}\in \mathbb{Q}$ is such that
		$\frac{\sigma_+}{\sigma_-}=\frac{\xi_+}{\xi_-}$ where  $\gcd(\xi_+, \xi_-) = 1$, and $\xi_+, \xi_-$ are even and odd 
		integers, respectively, or the other way around.\\
		$\mathbf{(H_{POL})}$ Assume that	$\frac{\sigma_+}{\sigma_-}$  is an irrational number. Then, suppose that there exists $w(\frac{\sigma_+}{\sigma_-} )\geq 2$, depending on 	$\frac{\sigma_+}{\sigma_-}$ such that for all sequences $\Gamma=(\xi_{1,n},\xi_{2,n})_{n\in \mathbb{N}}\in (\mathbb{N}\times\mathbb{N}^{\ast})^{\mathbb{N}}$ with $\xi_{1,n}\sim \xi_{2,n}$ for sufficiently large $n$, there exists a positive constant $c(\frac{\sigma_+}{\sigma_-},\Gamma)$ and a positive integer $N(\frac{\sigma_+}{\sigma_-},\Gamma)$, depending on $\frac{\sigma_+}{\sigma_-}$ and the sequence $\Gamma$ such that
		$$\left|\frac{\sigma_+}{\sigma_-}-\frac{\xi_{1,n}}{\xi_{2,n}}\right|>\frac{c(\frac{\sigma_+}{\sigma_-},\Gamma)}{\xi_{2,n}^{w(\frac{\sigma_+}{\sigma_-})}}, \forall n\geq N(\frac{\sigma_+}{\sigma_-},\Gamma).$$
		If condition $\mathbf{(H_{EXP})}$ is satisfied, the system is exponentially stable. When condition $\mathbf{(H_{POL})}$ is met, the system exhibits polynomial stability with a decay rate of type $t^{-\frac{2}{(4w\left(\frac{\sigma_+}{\sigma_-}\right) - 4)}}$. It is noteworthy that if the elastic part is replaced by heat without any damping, resulting in a piezoelectric-heat transmission problem, the system loses its exponential stability. Instead, it becomes polynomially stable with an energy decay rate of type $t^{-1}$ under condition $\mathbf{(H_{EXP})}$ and $t^{-\frac{2}{(4w\left(\frac{\sigma_+}{\sigma_-}\right) - 6)}}$ when condition $(\mathbf{(H_{POL})})$ is satisfied. A proof similar to that in \cite{AkilOzer2024} can employed to establish these decay rates.

	\end{enumerate}

\end{rem}


\begin{thebibliography}{10}
	
	\bibitem{Adhikari2009}
	S.~Adhikari, M.~I. Friswell, and D.~J. Inman.
	\newblock Piezoelectric energy harvesting from broadband random vibrations.
	\newblock {\em Smart Materials and Structures}, 18(11):115005, Sept. 2009.
	
	\bibitem{Afilal2023}
	M.~Afilal, A.~Soufyane, and M.~de~Lima~Santos.
	\newblock Piezoelectric beams with magnetic effect and localized damping.
	\newblock {\em Mathematical Control and Related Fields}, 13(1):250--264, 2023.
	
	\bibitem{Akil2022}
	M.~Akil.
	\newblock Stability of piezoelectric beam with magnetic effect under (coleman
	or pipkin)--gurtin thermal law.
	\newblock {\em Zeitschrift f{\"u}r angewandte Mathematik und Physik},
	73(6):236, Oct. 2022.
	
	\bibitem{AkilChitour2020}
	M.~Akil, Y.~Chitour, M.~Ghader, and A.~Wehbe.
	\newblock Stability and exact controllability of a timoshenko system with only
	one fractional damping on the boundary.
	\newblock {\em Asymptotic Analysis}, 119:221--280, 2020.
	\newblock 3-4.
	
	\bibitem{AkilGhader2020}
	M.~Akil, M.~Ghader, and A.~Wehbe.
	\newblock The influence of the coefficients of a system of wave equations
	coupled by velocities on its stabilization.
	\newblock {\em SeMA Journal}, 78(3):287–333, Nov. 2020.
	
	\bibitem{AkilMCRFIssa}
	M.~Akil, I.~Issa, and A.~Wehbe.
	\newblock Energy decay of some boundary coupled systems involving wave\
	euler-bernoulli beam with one locally singular fractional kelvin-voigt
	damping.
	\newblock {\em Mathematical Control and Related Fields}, 13(1):330--381, 2023.
	
	\bibitem{AkilOzer2024}
	M.~Akil, S.~Nicaise, A.~Özkan Özer, and V.~Régnier.
	\newblock Stability results for novel serially-connected magnetizable
	piezoelectric and elastic smart-system designs, to appear in Applied Mathematics and Optimization, 2024.
	
	\bibitem{AkilSoufyane2023}
	M.~Akil, A.~Soufyane, and Y.~Belhamadia.
	\newblock Stabilization results of a piezoelectric beams with partial viscous
	dampings and under lorenz gauge condition.
	\newblock {\em Applied Mathematics and Optimization}, 87(2), jan 2023.
	
	\bibitem{AkilWehbe2019MCRF}
	M.~Akil and A.~Wehbe.
	\newblock Stabilization of multidimensional wave equation with locally boundary
	fractional dissipation law under geometric conditions.
	\newblock {\em Mathematical Control and Related Fields}, 9(1):97--116, 2019.
	
	\bibitem{an2022stability}
	Y.~An, W.~Liu, and A.~Kong.
	\newblock Stability of piezoelectric beams with magnetic effects of fractional
	derivative type and with/without thermal effects, 2022.
	
	\bibitem{Ramos2018}
	A.~d. Araújo~Ramos, C.~Gonçalves, and S.~Neto.
	\newblock Exponential stability and numerical treatment for piezoelectric beams
	with magnetic effect.
	\newblock {\em ESAIM: Mathematical Modelling and Numerical Analysis}, 52, 03
	2018.
	
	\bibitem{Arendt-Batty}
	W.~Arendt and C.~J.~K. Batty.
	\newblock Tauberian theorems and stability of one-parameter semigroups.
	\newblock {\em Transactions of the American Mathematical Society},
	306(2):837--852, 1988.
	
	\bibitem{Ronald}
	R.~L. Bagley and P.~J. Torvik.
	\newblock Fractional calculus - a different approach to the analysis of
	viscoelastically damped structures.
	\newblock {\em AIAA Journal}, 21(5):741--748, 1983.
	
	\bibitem{RL}
	R.~L. Bagley and P.~J. Torvik.
	\newblock A theoretical basis for the application of fractional calculus to
	viscoelasticity.
	\newblock {\em Journal of Rheology}, 27(3):201--210, 1983.
	
	\bibitem{Batty01}
	C.~J.~K. Batty and T.~Duyckaerts.
	\newblock Non-uniform stability for bounded semi-groups on {B}anach spaces.
	\newblock {\em J. Evol. Equ.}, 8(4):765--780, 2008.
	
	\bibitem{Baur2014}
	C.~Baur, D.~J. Apo, D.~Maurya, S.~Priya, and W.~Voit.
	\newblock {\em Advances in Piezoelectric Polymer Composites for Vibrational
		Energy Harvesting}, chapter~1, pages 1--27.
	
	\bibitem{Benaissa2019}
	A.~Benaissa and S.~Gaouar.
	\newblock Asymptotic stability for the lam\'e system with fractional boundary
	damping.
	\newblock {\em Computers and Mathematics with Applications}, 77(5):1331–1346,
	Mar. 2019.
	
	\bibitem{Borichev01}
	A.~Borichev and Y.~Tomilov.
	\newblock Optimal polynomial decay of functions and operator semigroups.
	\newblock {\em Math. Ann.}, 347(2):455--478, 2010.
	
	\bibitem{C1}
	M.~Caputo.
	\newblock Linear models of dissipation whose q is almost frequency
	independent—ii.
	\newblock {\em Geophysical Journal International}, 13(5):529--539, 11 1967.
	
	\bibitem{C2}
	M.~Caputo.
	\newblock {\em Elasticit{\`a} e dissipazione}.
	\newblock Zanichelli, 1969.
	
	\bibitem{Caputo-Fab}
	M.~Caputo and M.~Fabrizio.
	\newblock A new definition of fractional derivative without singular kernel.
	\newblock {\em Prog Fract Differ Appl}, 1:73--85, 04 2015.
	
	\bibitem{C3}
	M.~Caputo and F.~Mainardi.
	\newblock Linear models of dissipation in anelastic solids.
	\newblock {\em La Rivista del Nuovo Cimento}, 1(2):161–198, Apr. 1971.
	
	\bibitem{Castille}
	C.~Castille, I.~Dufour, and C.~Lucat.
	\newblock Longitudinal vibration mode of piezoelectric thick-film
	cantilever-based sensors in liquid media.
	\newblock {\em Applied Physics Letters}, 96:154102, 04 2010.
	
	\bibitem{CM}
	U.~J. Choi and R.~MacCamy.
	\newblock Fractional order volterra equations with applications to elasticity.
	\newblock {\em Journal of Mathematical Analysis and Applications},
	139(2):448--464, 1989.
	
	\bibitem{Lions1985}
	R.~Dautray and J.~L. Lions.
	\newblock {\em Analyse Mathématique et Calcul Numérique Pour les Sciences et
		les Techniques}, volume~3.
	\newblock 1985.
	
	\bibitem{Destuynder1992}
	P.~Destuynder, I.~Legrain, L.~Castel, and N.~Richard.
	\newblock Theoretical, numerical and experimental discussion on the use of
	piezoelectric devices for control-structure interaction.
	\newblock {\em European Journal of Mechanics A-solids}, 11:181--213, 1992.
	
	\bibitem{JDEOzer2021}
	M.~Freitas, A.~Ramos, A.~Özer, and D.~{Almeida Júnior}.
	\newblock Long-time dynamics for a fractional piezoelectric system with
	magnetic effects and fourier's law.
	\newblock {\em Journal of Differential Equations}, 280:891--927, 2021.
	
	\bibitem{Banks1996}
	Y.~W. H.~T.~Banks, R. C.~Smith.
	\newblock {\em Smart material structures: Modelling, Estimation and Control}.
	\newblock 1996.
	
	\bibitem{Kapitonov2007}
	B.~Kapitonov, B.~Miara, and G.~Menzala.
	\newblock Boundary observation and exact control of a quasi-electrostatic
	piezoelectric system in multilayered media.
	\newblock {\em SIAM J. Control and Optimization}, 46:1080--1097, 01 2007.
	
	\bibitem{KST}
	A.~Kilbas, H.~Srivastava, and J.~Trujillo.
	\newblock {\em Theory And Applications of Fractional Differential Equations}.
	\newblock North-Holland Mathematics Studies. Elsevier Science \& Tech, 2006.
	
	\bibitem{Komornik1994}
	V.~Komornik.
	\newblock {\em Exact Controllability and Stabilization: The Multiplier Method}.
	\newblock 1994.
	
	\bibitem{LASIECKA2009167}
	I.~Lasiecka and B.~Miara.
	\newblock Exact controllability of a 3d piezoelectric body.
	\newblock {\em Comptes Rendus Mathematique}, 347(3):167--172, 2009.
	
	\bibitem{RaoLiu01}
	Z.~Liu and B.~Rao.
	\newblock Characterization of polynomial decay rate for the solution of linear
	evolution equation.
	\newblock {\em Z. Angew. Math. Phys.}, 56(4):630--644, 2005.
	
	\bibitem{LiuZheng01}
	Z.~Liu and S.~Zheng.
	\newblock {\em Semigroups associated with dissipative systems}, volume 398 of
	{\em Chapman \& Hall/CRC Research Notes in Mathematics}.
	\newblock Chapman \& Hall/CRC, Boca Raton, FL, 1999.
	
	\bibitem{Bonetti}
	M.~Mainardi and E.~Bonetti.
	\newblock The application of real-order derivatives in linear viscoelasticity.
	\newblock pages 64--67, 1988.
	
	\bibitem{Maryati2019}
	T.~Maryati, J.~Muñoz~Rivera, V.~Poblete, and O.~Vera.
	\newblock Asymptotic behavior in a laminated beams due interfacial slip with a
	boundary dissipation of fractional derivative type.
	\newblock {\em Applied Mathematics and Optimization}, 84(1):85–102, Nov.
	2019.
	
	\bibitem{app}
	D.~Matignon.
	\newblock Asymptotic stability of webster-lokshin equation.
	\newblock {\em Mathematical Control and Related Fields}, 4:481, 2014.
	
	\bibitem{15}
	B.~Mbodje.
	\newblock Wave energy decay under fractional derivative controls.
	\newblock {\em IMA Journal of Mathematical Control and Information},
	23(2):237--257, 06 2006.
	
	\bibitem{Ozer2014}
	K.~Morris and A.~Ozer.
	\newblock Modeling and stabilizability of voltage-actuated piezoelectric beams
	with magnetic effects.
	\newblock {\em SIAM J. Control and Optim.}, 52:2371--2398, 07 2014.
	
	\bibitem{Ozer2013}
	K.~Morris and A.~O. Ozer.
	\newblock Strong stabilization of piezoelectric beams with magnetic effects.
	\newblock In {\em 52nd IEEE Conference on Decision and Control}, pages
	3014--3019, 2013.
	
	\bibitem{article2017}
	V.~Nguyen, N.~Wu, and Q.~Wang.
	\newblock A review on energy harvesting from ocean waves by piezoelectric
	technology.
	\newblock {\em Journal of Modeling in Mechanics and Materials}, 1, 01 2017.
	
	\bibitem{Oh2010}
	S.~J. Oh, H.~Han, S.~Han, J.~Lee, and W.~Chun.
	\newblock Development of a tree‐shaped wind power system using piezoelectric
	materials.
	\newblock {\em International Journal of Energy Research}, 34:431 -- 437, 04
	2010.
	
	\bibitem{ozer2023maximal}
	A.~O. Ozer, R.~Emran, and A.~K. Aydin.
	\newblock Maximal decay rate and optimal sensor feedback amplifiers for fast
	stabilization of magnetizable piezoelectric beam equations, 2023.
	
	\bibitem{ozer2023exponential}
	A.~O. Ozer, I.~Khalilullah, and U.~Rasaq.
	\newblock The exponential stabilization of a heat and piezoelectric beam
	interaction with static or hybrid feedback controllers, 2023.
	
	\bibitem{Pazy}
	A.~Pazy.
	\newblock {\em Semigroups of Linear Operators and Applications to Partial
		Differential Equations}.
	\newblock Applied mathematical sciences. Springer, 1983.
	
	\bibitem{SKM}
	S.~G. Samko, A.~A. Kilbas, and O.~I. Marichev.
	\newblock Fractional integrals and derivatives: Theory and applications.
	\newblock 1993.
	
	\bibitem{Yang2004AnIT}
	J.~shi Yang.
	\newblock {\em An Introduction to the Theory of Piezoelectricity}.
	\newblock 2004.
	
	\bibitem{Smith2005}
	R.~C. Smith.
	\newblock {\em Smart Material Systems: Model Development}.
	\newblock Society for Industrial and Applied Mathematics, Jan. 2005.
	
	\bibitem{Soufyane2021}
	A.~Soufyane, A.~Mounir, and M.~Santos.
	\newblock Energy decay for a weakly nonlinear damped piezoelectric beams with
	magnetic effects and a nonlinear delay term.
	\newblock {\em Zeitschrift für angewandte Mathematik und Physik}, 72, 08 2021.
	
	\bibitem{Tiersten1969}
	H.~F. Tiersten.
	\newblock {\em Linear Piezoelectric Plate Vibrations}.
	\newblock 1969.
	
	\bibitem{peter}
	P.~Torvik and R.~Bagley.
	\newblock On the appearance of the fractional derivative in the behavior of
	real materials.
	\newblock {\em Journal of Applied Mechanics}, 51, 06 1984.
	
	\bibitem{Williams1996}
	C.~Williams and R.~Yates.
	\newblock Analysis of a micro-electric generator for microsystems.
	\newblock {\em Sensors and Actuators A: Physical}, 52(1–3):8–11, Mar. 1996.
	
	\bibitem{Wu2015}
	N.~Wu, Q.~Wang, and X.~Xie.
	\newblock Ocean wave energy harvesting with a piezoelectric coupled buoy
	structure.
	\newblock {\em Applied Ocean Research}, 50:110–118, Mar. 2015.
	
	\bibitem{Yang2006}
	J.~Yang.
	\newblock A review of a few topics in piezoelectricity.
	\newblock {\em Applied Mechanics Reviews - APPL MECH REV}, 59, 11 2006.
	
	\bibitem{zhang_zuazua_2007}
	X.~Zhang and E.~Zuazua.
	\newblock Long-time behavior of a coupled heat-wave system arising in
	fluid-structure interaction.
	\newblock {\em Archive for Rational Mechanics and Analysis}, 184(1):49--120,
	2007.
	
\end{thebibliography}

\end{document}